\documentclass[a4paper,reqno,oneside]{amsart}

\usepackage[T1]{fontenc}
\usepackage[utf8]{inputenc}
\usepackage{lmodern}
\usepackage{microtype}

\usepackage[left=80pt,right=80pt,top=80pt,bottom=94pt,heightrounded]{geometry}

\usepackage{amssymb}
\usepackage{stmaryrd} % for \mapsfrom
\usepackage{mathtools}

\usepackage{hyperref}
\usepackage[hyperpageref]{backref}
\usepackage[nobysame,alphabetic,initials]{amsrefs}

\DefineSimpleKey{bib}{how}
\DefineSimpleKey{bib}{mrclass}
\DefineSimpleKey{bib}{mrnumber}
\DefineSimpleKey{bib}{fjournal}
\DefineSimpleKey{bib}{mrreviewer}

\renewcommand{\PrintDOI}[1]{%
  \href{http://dx.doi.org/#1}{{\tt DOI:#1}}%
}
\renewcommand{\eprint}[1]{#1}
\BibSpec{book}{%
    +{}  {\PrintPrimary}                {transition}
    +{.} { \PrintDate}                  {date}
    +{.} { \textit}                     {title}
    +{.} { }                            {part}
    +{:} { \textit}                     {subtitle}
    +{,} { \PrintEdition}               {edition}
    +{}  { \PrintEditorsB}              {editor}
    +{,} { \PrintTranslatorsC}          {translator}
    +{,} { \PrintContributions}         {contribution}
    +{,} { }                            {series}
    +{,} { \voltext}                    {volume}
    +{,} { }                            {publisher}
    +{,} { }                            {organization}
    +{,} { }                            {address}
    +{,} { }                            {status}
    +{,} { \PrintDOI}                   {doi}
    +{,} { \PrintISBNs}                 {isbn}
    +{}  { \parenthesize}               {language}
    +{}  { \PrintTranslation}           {translation}
    +{;} { \PrintReprint}               {reprint}
    +{.} { }                            {note}
    +{.} {}                             {transition}
}
\BibSpec{article}{%
    +{}  {\PrintAuthors}                {author}
    +{,} { \textit}                     {title}
    +{.} { }                            {part}
    +{:} { \textit}                     {subtitle}
    +{,} { \PrintContributions}         {contribution}
    +{.} { \PrintPartials}              {partial}
    +{,} { }                            {journal}
    +{}  { \textbf}                     {volume}
    +{}  { \PrintDatePV}                {date}
    +{,} { \issuetext}                  {number}
    +{,} { \eprintpages}                {pages}
    +{,} { }                            {status}
    +{,} { \PrintDOI}                   {doi}
    +{,} { \eprint}        {eprint}
    +{}  { \parenthesize}               {language}
    +{}  { \PrintTranslation}           {translation}
    +{;} { \PrintReprint}               {reprint}
    +{.} { }                            {note}
    +{.} {}                             {transition}
}
\BibSpec{collection.article}{%
    +{}  {\PrintAuthors}                {author}
    +{,} { \textit}                     {title}
    +{.} { }                            {part}
    +{:} { \textit}                     {subtitle}
    +{,} { \PrintContributions}         {contribution}
    +{,} { \PrintConference}            {conference}
    +{}  {\PrintBook}                   {book}
    +{,} { }                            {booktitle}
    +{,} { \PrintDateB}                 {date}
    +{,} { pp.~}                        {pages}
    +{,} { }                            {publisher}
    +{,} { }                            {organization}
    +{,} { }                            {address}
    +{,} { }                            {status}
    +{,} { \PrintDOI}                   {doi}
    +{,} { \eprint}        {eprint}
    +{}  { \parenthesize}               {language}
    +{}  { \PrintTranslation}           {translation}
    +{;} { \PrintReprint}               {reprint}
    +{.} { }                            {note}
    +{.} {}                             {transition}
}
\BibSpec{misc}{%
  +{}{\PrintAuthors}  {author}
  +{,}{ \textit}      {title}
  +{.}{ }             {how}
  +{}{ \parenthesize} {date}
  +{,} { available at \eprint}        {eprint}
  +{,}{ available at \url}{url}
  +{,}{ }             {note}
  +{.}{}              {transition}
}

\theoremstyle{plain}
\newtheorem{theo}{Theorem}[section]%
\newtheorem{prop}[theo]{Proposition}%
\newtheorem{coro}[theo]{Corollary}%
\newtheorem{lemm}[theo]{Lemma}%
\newtheorem{theoremA}{Theorem}

\theoremstyle{definition}%
\newtheorem{defi}[theo]{Definition}%
\theoremstyle{remark}%
\newtheorem{rema}[theo]{Remark}%
\newtheorem{exem}[theo]{Example}%

\usepackage{tensor}

\usepackage{tikz}
\usetikzlibrary{cd}

\DeclareRobustCommand{\SkipTocEntry}[5]{}

\mathchardef\mhyph="2D 				% myth hyphen
\newcommand{\numberset}{\mathbb} 
\newcommand{\N}{\numberset{N}} 
\newcommand{\Z}{\numberset{Z}} 
\newcommand{\Q}{\numberset{Q}} 
\newcommand{\R}{\numberset{R}}
\newcommand{\bC}{\numberset{C}}

\newcommand{\bH}{\mathbb{H}}

\newcommand{\ep}{\epsilon}

\newcommand{\cE}{\mathcal{E}}

\newcommand{\cG}{\mathcal{G}}

\newcommand{\cI}{\mathcal{I}}
\newcommand{\cK}{\mathcal{K}}
\newcommand{\cL}{\mathcal{L}}
\newcommand{\cM}{\mathcal{M}}
\newcommand{\cN}{\mathcal{N}}
\newcommand{\cO}{\mathcal{O}}
\newcommand{\cP}{\mathcal{P}}

\newcommand{\cS}{\mathcal{S}}
\newcommand{\cT}{\mathcal{T}}
\newcommand{\I}{\mathcal{I}}

\newcommand{\PI}{\langle \cP_\I\rangle}

\newcommand{\Ab}{\mathrm{Ab}}

\newcommand{\id}{\mathrm{id}}

\newcommand{\absv}[1]{\left| #1 \right|}

\newcommand{\GCalg}{\mathrm{C}^*_G}
\newcommand{\HCalg}{\mathrm{C}^*_H}

\newcommand{\medwedge}{{\textstyle\bigwedge}}

\DeclareMathOperator{\Tor}{Tor}

\DeclareMathOperator{\KKK}{KK}
\DeclareMathOperator{\BF}{BF}
\DeclareMathOperator{\End}{End}

\DeclareMathOperator{\cok}{cok}
\DeclareMathOperator{\Tot}{Tot}

\DeclareMathOperator{\GL}{GL}
\DeclareMathOperator{\Hom}{Hom}

\let\Res\relax

\DeclareMathOperator{\Ind}{Ind}

\DeclareMathOperator{\Res}{Res}

\newcommand{\tens}[2]{%
  \mathbin{\tensor*[^#1]{\otimes}{_{#2}}}%
}
\newcommand{\tensb}[3]{%
  \mathbin{\tensor*[^#1]{\otimes}{_{#2}^#3}}%
}

\author{Valerio Proietti}
\address{Research Center for Operator Algebras, Department of Mathematics, and Shanghai Key Laboratory of Pure Mathematics and Mathematical Practice, East China Normal University, Shanghai 200241, China}
\email{proiettivalerio@math.ecnu.edu.cn}

\author{Makoto Yamashita}
\address{Department of Mathematics, University of Oslo, P.O. box 1053, Blindern, 0316 Oslo, Norway}
\email{makotoy@math.uio.no}

\date{November 17, 2021: minor changes; May 16, 2021}

\title[Homology and $K$-theory of dynamical systems, II]{Homology and $K$-theory of dynamical systems\\
II. Smale spaces with totally disconnected transversal}

\begin{document}
%******************************************************************
% Beginning
%******************************************************************

\begin{abstract}
We apply our previous work on the relation between groupoid homology and $K$-theory to Smale spaces.
More precisely, we consider the unstable equivalence relation of a Smale space with totally disconnected stable sets, and prove that the associated spectral sequence shows Putnam's stable homology groups on the second sheet.
Moreover, this homology is in fact isomorphic to groupoid homology of the unstable equivalence relation.
\end{abstract}

\subjclass[2010]{46L85; 19K35, 37D99}
\keywords{groupoid, C$^*$-algebra, $K$-theory, homology, Baum--Connes conjecture, Smale space.}

\maketitle
\setcounter{tocdepth}{1}
\tableofcontents

%******************************************************************
% Main
%******************************************************************

\section*{Introduction}

Continuing our previous work \cite{valmako:part1} on groupoid homology and $K$-theory for ample groupoids, in this paper we look at the groupoids arising from Smale spaces, and Putnam's homology \cite{put:HoSmale}.

The framework of Smale spaces was introduced by Ruelle \cite{ruelle:thermo}, who designed them to model the basic sets of Axiom A diffeomorphisms \cite{smale:A}. This turned out to be a particularly nice class of hyperbolic topological dynamical systems, where Markov partitions provide a symbolic approximation of the dynamics.

Groupoids with totally disconnected base (ample groupoids) arise from Smale spaces with totally disconnected stable sets. This is especially useful in the study of dynamical systems whose topological dimension is not zero, but whose dynamics is completely captured by restricting to a totally disconnected transversal.
Such spaces include generalized solenoids \citelist{\cite{thom:sol}\cite{will:exp}}, substitution tiling spaces \cite{putand:til}*{Theorem 3.3}, dynamical systems of self-similar group actions \cite{nekra:crelle}, and can be characterized as certain inverse limits \cite{wie:inv}.

Beyond the theory of dynamical systems, these groupoids also play an important role in the theory of operator algebras, where they provide an invaluable source of examples of C$^*$-algebras.
These are obtained by considering the (reduced) groupoid C$^*$-algebras $C^*_r(G)$ \cite{ren:group}, generalizing the crossed product algebras for group actions on the Cantor set.
The resulting C$^*$-algebras capture interesting aspects of the homoclinic and heteroclinic structure of expansive dynamics \citelist{\cite{put:algSmale}\cite{thomsen:smale}\cite{mats:ruellemarkov}}, extending the correspondence between shifts of finite type and the Cuntz--Krieger algebras.

In \cite{valmako:part1}, based on the Meyer--Nest theory of triangulated categorical structure on equivariant$\KKK$-theory \citelist{\cite{nestmeyer:loc}\cite{meyernest:tri}\cite{meyer:tri}}, we constructed a spectral sequence of the form
\begin{equation}\label{eq:part-I-main-res-formula}
E^2_{p q} = H_p(G, K_q(\bC)) \Rightarrow K_{p+q}(C^*_r(G))
\end{equation}
for étale groupoids $G$ that have torsion-free stabilizers and satisfy a strong form of the Baum--Connes conjecture as in Tu's work \cite{tu:moy}.
While this directly applies to the reduction of the unstable equivalence relation to transversals of Smale spaces as above, there is another homology theory proposed by Putnam \cite{put:HoSmale}. We show that one of the variants, $H^s_\bullet$, fits into this scheme for the groupoid $R^u(Y, \psi)$ of the unstable equivalence relation on the underlying space, as follows.

\begin{theoremA}[Theorem \ref{thm:putnam-homology-convergence}]\label{thm:b}
Let $(Y,\psi)$ be an irreducible Smale space with totally disconnected stable sets, and $R^u(Y,\psi)$ be the groupoid of the unstable equivalence relation.
Then there is a convergent spectral sequence
\begin{equation*}
E^2_{pq} = E^3_{pq} = H_p^s(Y,\psi)\otimes K_q(\bC) \Rightarrow K_{p+q}(C^*(R^u(Y,\psi))).
\end{equation*}
\end{theoremA}

This result gives a partial answer to a question raised by Putnam, who asked to relate the $K$-theory of $C^*(R^u(Y,\psi))$ to his homology theory $H_\bullet^s(Y,\psi)$ \cite{put:HoSmale}*{Question 8.4.1}.
An immediate consequence is that the $K$-groups of $C^*(R^u(Y,\psi))$ are of finite rank. Our proof of Theorem \ref{thm:b} above is based on a ``relativized'' analogue of the method we developed in \cite{valmako:part1}.

Although we give an independent proof of Theorem \ref{thm:b}, it can also be obtained from the spectral sequence \eqref{eq:part-I-main-res-formula} and the result below.
\begin{theoremA}[Theorem \ref{thm:putnam-homology-is-groupoid-homology}] 
For any étale groupoid $G$ that is Morita equivalent to $R^u(Y,\psi)$, we have an isomorphism $H^s_p(Y, \psi)\simeq H_p(G, \Z)$.
\end{theoremA}

In order to prove the result above, we turn the definition of Putnam's homology into a resolution of modules which computes groupoid homology.
As a corollary we obtain a Künneth formula for $H^s_\bullet$, generalizing a result in \cite{MR3576278}.
In the framework of substitution tiling spaces \cite{putand:til}, this result, combined with those of \cite{valmako:part1}, implies that $H_\bullet^s(\Omega, \omega)$ for the associated Smale space $(\Omega, \omega)$ is isomorphic to the degree shift of the Čech cohomology of $\Omega$ (this partially solves \cite{put:HoSmale}*{Question 8.3.2}).

\medskip
This paper is organized as follows.
In Section \ref{sec:prelim} we lay out the basic notation and definitions for all the background objects of the paper.

In Section \ref{sec:pullback-res-groupoids}, we discuss the multiple fibered product of groupoid homomorphisms, generalizing a construction in \cite{cramo:hom}.
This provides the spatial implementation of the groupoid bar complex in the case of the inclusion map $G^{(0)} \to G$ regarded as a groupoid homomorphism.
Turning to Smale spaces, a key technical transversality result in Proposition~ \ref{prop:transversality} allows us to relate multiple fiber products of $s$-bijective maps $f\colon (\Sigma,\sigma) \to (Y,\psi)$ from shifts of finite type to the multiple groupoid fibered products.
This is a crucial ingredient for Putnam's homology.

In Section \ref{sec:approx-equivar-KK}, we consider a simplicial object relating homology and $K$-groups of the groupoid following the scheme of \cite{valmako:part1}.
We recall that the spectral sequence \eqref{eq:part-I-main-res-formula} appeared from the Moore complex of the simplicial object $(F(L^{n+1} A))_{n=0}^\infty$ with $L = \Ind^G_X \Res^G_X$, converging to $F(A)$, with the functor $F = K_\bullet(G \ltimes \mhyph)$.
For the groupoid of the unstable equivalence relation on a Smale space $(Y, \psi)$ with totally disconnected stable sets, we follow the same scheme, but replace $X$ by the subgroupoid coming from an $s$-bijective factor map from a shift of finite type.
The resulting complex is isomorphic to the one defining Putnam's homology~$H^s_\bullet(Y, \psi)$.

In Section \ref{sec:compar-homology} we combine the previous sections and prove our main results.

Next in Section \ref{sec:k-th-ruelle-alg}, we explain how a similar method applies to the K-groups of Ruelle algebras.
Here we obtain a homology built out of Bowen--Franks groups closely following Putnam's homology theory.
In Section \ref{sec:examples} we discuss some examples, including solenoids and self-similar group actions.

\addtocontents{toc}{\SkipTocEntry}\subsection*{Acknowledgements}

We are indebted to R.~Nest for proposing the topic of this paper as a research project, and for numerous stimulating conversations.
We are also grateful to R.~Meyer for valuable advice concerning equivariant $K$-theory and for his careful reading of our draft.
Thanks also to M.~Dadarlat, R.~Deeley, M.~Goffeng, V.~Nekrashevych, and I.~F.~Putnam for stimulating conversations and encouragement at various stages, which led to numerous improvements.
We would also like to thank the anonymous reviewer for careful reading and useful suggestions.

This research was partly supported through V.P.'s ``Oberwolfach Leibniz Fellowship'' by the \emph{Mathematisches Forschungsinstitut Oberwolfach} in 2020. In addition, V.P. was supported by the Science and Technology Commission of Shanghai Municipality (grant No.~18dz2271000). M.Y. acknowledges support by Grant for Basic Science Research Projects from The Sumitomo Foundation at an early stage of collaboration.

\section{Preliminaries}\label{sec:prelim}

In this section we fix conventions and go over some preliminaries in order to clarify the basic notions that will be used in the rest of the paper. We generally follow our exposition in \cite{valmako:part1}.

\subsection{Locally compact groupoids}

Throughout the paper $G$ denotes a topological groupoid with unit space $X = G^{(0)}$.
We let $s,r\colon G\to X$ denote respectively the source and range maps. 
In addition, we let $G_x=s^{-1}(x)$, $G^x=r^{-1}(x)$, and for a subset $A \subset X$, we write $G_A = \bigcup_{x \in A} G_x$, $G^A = \bigcup_{x \in A} G^x$, and $G|_A = G^A \cap G_A$.

\begin{defi}
A topological groupoid $G$ is \emph{étale} if $s$ and $r$ are local homeomorphisms, and \emph{ample} if it is étale and $G^{(0)}$ is totally disconnected.
\end{defi}

If $G$ is étale and $g \in G$, then by definition, for small enough neighborhoods $U$ of $s(g)$ there is a neighborhood $U'$ of $g$ such that $s(U') = U$, and the restriction of $s$ and $r$ to $U'$ are homeomorphisms onto the images. When this is the case, we write $g(U) = r(U')$ and use $g$ as a label for the map $U \to g(U)$ induced by the identification of $U \sim U' \sim g(U)$.

As our blanket assumption, we further assume that a topological groupoid is second countable, locally compact Hausdorff, and admits a continuous Haar system $\lambda = (\lambda^x)_{x \in X}$ (i.e., an invariant continuous distribution of Radon measures on the spaces $(G^x)_{x \in X}$), so that its full and reduced groupoid C$^*$-algebras $C^*(G, \lambda)$, $C^*_r(G, \lambda)$ make sense. In general these algebras may depend on $\lambda$, though different Haar systems lead to Morita equivalent C$^*$-algebras (see also \cite{valmako:part1}*{Remark 1.7}).
In particular, $G$ and $X$ are $\sigma$-compact and paracompact. When dealing with étale groupoids, we take as usual the counting measure on $G^x$, and in this case we suppress the notation $\lambda$, and simply write $C^*_r(G)$ instead of $C^*_r(G, \lambda)$. For a C$^*$-algebra $A$ equipped with an action of $G$, the (reduced) crossed product of $A$ by $G$ in this paper will be equivalently denoted by $A\rtimes G$, $G\ltimes A$, or $C^*_r(G,A)$ depending on what is easier to read in context.

A locally compact groupoid $G$ is \emph{amenable} if there exists a net of probability measures on $G^x$ for $x \in G^{(0)}$ which is approximately invariant, see \cite{renroch:amgrp}.
In this case we have $C^*(G, \lambda) = C^*_r(G, \lambda)$ for any Haar system $\lambda$.
This covers all of our concrete examples.

\smallskip
The notion of Morita equivalence of groupoids in the sense of \cite{murewi:morita} plays an important role in this paper. We review it here below for convenience. First, recall a topological groupoid $G$ is \emph{proper} if the map $(r\times s)\colon G \to X\times X$ is proper. Furthermore, if $Z$ is a locally compact, Hausdorff $G$-space, we say that $G$ \emph{acts properly} on $Z$ if the transformation groupoid $G \ltimes Z$ is proper. The map $Z\to G^{(0)}$ underlying the $G$-action is called the \emph{anchor map}.

\begin{defi}
The groupoids $G$ and $H$ are \emph{Morita equivalent} if there is a locally compact Hausdorff space $Z$ such that
\begin{itemize}
\item $Z$ is a free and proper left $G$-space with anchor map $\rho\colon Z \to G^{(0)}$;
\item $Z$ is a free and proper right $H$-space with anchor map $\sigma\colon Z \to H^{(0)}$;
\item the actions of $G$ and $H$ on $Z$ commute;
\item $\rho\colon Z \to G^{(0)}$ induces a homeomorphism $Z/ H \to G^{(0)}$;
\item $\sigma\colon Z \to H^{(0)}$ induces a homeomorphism $G\backslash Z \to H^{(0)}$.
\end{itemize}
\end{defi}

This can be conveniently packaged by a \emph{bibundle} over $G$ and $H$: that is, a topological space $Z$ with $G$ and $H$ acting continuously from both sides with surjective and open anchor maps, such that that the maps
\begin{align*}
G \times_{G^{(0)}} Z &\to Z \times_{H^{(0)}} Z, \quad (g, z) \mapsto (g z, z), &
Z \times_{H^{(0)}} H &\to Z \times_{G^{(0)}} Z, \quad (z, h) \mapsto (z, z h)
\end{align*}
are homeomorphisms.

An important class of Morita equivalences comes from generalized transversals. For a topological space $X$ and $x \in X$, let us denote the family of the open neighborhoods of $x$ by $\cO(x)$.

\begin{defi}[\cite{put:spiel}]
\label{def:gen-transv}
Let $G$ be a topological groupoid.
A \emph{generalized transversal} for $G$ is given by a topological space $T$ and an injective continuous map $f\colon T \to G^{(0)}$ such that:
\begin{itemize}
\item $f(T)$ meets every orbit of $G$; and
\item the \emph{condition (Ar)} for neighborhoods of $x$ and $f^{-1}(r x)$ holds for all $x \in G$, i.e.,
\begin{gather*}
\forall x \in G^{f(T)}, U_0 \in \cO(x), V_0 \in \cO(f^{-1}(r x))\quad \exists  U \in \cO(x), V \in \cO(f^{-1}(r x))\colon\\
U \subset U_0, V \subset V_0, \forall y \in U \quad \exists !\, z \in U, s( y) = s( z), r( z) \in f(V).
\end{gather*}
\end{itemize}
\end{defi}

Under the above setting, there is a finer topology on the subgroupoid $H = G|_{f(T)}$ such that $H$ is étale and Morita equivalent to $G$ \cite{put:spiel}*{Theorem 3.6}.
The equivalence is implemented by the principal bibundle $G^{f(T)}$ with a natural finer topology from that of $G$ and $T$.

\subsection{Induction and restriction for groupoid \texorpdfstring{$\KKK$}{KK}-theory}
\label{sec:ind-res-adj}

Suppose $G$ is an étale groupoid as in the previous subsection.
We denote by $\KKK^G$ the category of separable $G$-C$^*$-algebras with the equivariant KK-groups $\KKK^G(A, B)$ \cite{gall:kk} as morphisms sets.

Let $H\subseteq G$ be an open subgroupoid with the same base space $X=G^{(0)}=H^{(0)}$.
In particular $H$ is an étale groupoid over $X$, and the restriction of action gives a functor $\Res^G_H\colon \KKK^G\to \KKK^H$.
It admits a left adjoint, which is an analogue of induction, as follows. Full details will appear elsewhere in a joint work of the first named author with C.~Bönicke.

Let $B$ be an $H$-C$^*$-algebra, with structure map $\rho\colon C_0(X)\to Z(\cM(B))$.
As before, take the $C_0(G)$-algebra
\[
B' = C_0(G) \tens{s}{C_0(X)} B,
\]
where the superscript $s$ indicates that we regard $C_0(G)$ as a $C_0(X)$-algebra with respect to the source map.
This has a right action of $H$, by combination of the right translation on $C_0(G)$ and the action on $B$ twisted by the inverse map of $H$.
We then set
\[
\Ind_H^G(B) = B'\rtimes H = (C_0(G) \tens{s}{C_0(X)} B)\rtimes_{\text{diag}} H.
\]
This can be regarded as the crossed product of $B'$ by the transformation groupoid $G \rtimes H$ for the right translation action of $H$ on $G$.
Moreover, notice that $G$ also acts on $B'$ by left translation on $C_0(G)$.
This induces a continuous action of $G$ on $\Ind_H^G(B)$.

Let $A$ be a $G$-C$^*$-algebra.
Then the Haar system on $G$ induces an $A$-valued inner product on $C_c(G) \otimes_{C_0(X)} A$, and by completion we obtain a right Hilbert $A$-module $E^G_A = L^2(G, A)$.
We then have the following, see Appendix \ref{sec:app-ind-ftr} for details.

\begin{prop}
\label{prop:Mor-equiv-A-IndGG-A}
Under the above setting, $E^G_A$ implements an equivariant strong Morita equivalence between $A$ and $\Ind^G_G A$.
\end{prop}

Let $\kappa$ denote the inclusion homomorphism
\[
\Ind_H^G\Res^G_H(A)=(C_0(G) \tens{s}{C_0(X)} A)\rtimes H \to (C_0(G) \tens{s}{C_0(X)} A)\rtimes G = \Ind^G_G A,
\] 
induced by $H \subseteq G$ because $H$ is open, and let $\iota$ denote the map
\[
\Ind^H_H B = (C_0(H)\tens{s}{C_0(X)} B)\rtimes H \to (C_0(G)\tens{s}{C_0(X)} B)\rtimes H=\Res^G_H\Ind_H^G(B)
\]
induced by the ideal inclusion $C_0(H)\subseteq C_0(G)$.

\begin{theo}\label{thm:iradj}
The functor $\Ind^G_H$ induces a functor $\KKK^H \to \KKK^G$, and there is a natural isomorphism
\[
\KKK^G(\Ind^G_H B, A) \simeq \KKK^H(B, \Res^G_H A)
\]
defining an adjunction $(\epsilon,\eta)\colon \Ind_H^G \dashv \Res^G_H$ with counit and unit natural morphisms
\begin{align*}
\epsilon_A &= [\kappa] \otimes_{\Ind^G_G A} [E^G_A] \in \KKK^G(\Ind_H^G\Res^G_H A, A),&
\eta_B &= [\bar E^{H}_B] \otimes_{\Ind^H_H B} [\iota] \in \KKK^H(B, \Res^G_H\Ind_H^G B).
\end{align*}
\end{theo}

\begin{exem}
If $G$ is the transformation groupoid $\Gamma \ltimes X$ and $H = X$, the previous theorem amounts to
\[
\KKK^{\Gamma \ltimes X}(C_0(\Gamma) \otimes B, A) \simeq \KKK^X(B, A)
\]
for any $C_0(X)$-algebra $B$ and $G$-algebra $A$, where the $\Gamma$-action on $C_0(\Gamma) \otimes B$ is given by translation on the factor $C_0(\Gamma)$.
\end{exem}

\subsection{Triangulated categories and spectral sequences}

Let us quickly recall the formalism of \cite{valmako:part1} behind the spectral sequence in \eqref{eq:part-I-main-res-formula}.
The main ingredients are triangulated categories $\cS$ and $\cT$ with countable direct sums, and exact functors $E \colon \cS \to \cT$ and $F \colon \cT \to \cS$ compatible with countable direct sums, with natural isomorphisms
\begin{equation*}
%\label{eq:adj-functors}
\cS(A, F B) \simeq \cT(E A, B) \quad (A \in \cS, B \in \cT).
\end{equation*}

Let $\cI$ denote the kernel of $F$, that is, the collection of morphisms $f$ in $\cS$ such that $F f = 0$.
An object $A \in \cT$ is said to be $\cI$-projective if any $f \in \cI(A', A'')$ induces the zero map $\cT(A, A') \to \cT(A, A'')$.
The category $\cT$ has two triangulated subcategories: the one $\langle E\cS \rangle$ generated by the image of $E$, and another $\cN_{\cI}$ consisting of the objects $N$ satisfying $F N = 0$.

Now, consider the endofunctor $L = E F$ on $\cT$.

\begin{prop}[\cite{valmako:part1}*{Proposition 3.1}]\label{prop:simplicial-res-from-adj-fs}
In the above setting, any object $A \in \cT$ admits an $\cI$-projective resolution $P_\bullet$ consisting of $P_n = L^{n+1} A$. The pair of subcategories $(\langle E\cS \rangle, \cN_\cI)$ is complementary.
\end{prop}

In particular, for any $A \in \cT$, there is an exact triangle
\[
P \to A \to N \to \Sigma P
\]
satisfying $P \in \langle E\cS \rangle$ and $N \in \cN_{\cI}$.
By \cite{meyer:tri}*{Theorems 4.3 and 5.1}, we then get the following.

\begin{theo}\label{thm:spseqtri}
Let $K\colon \cT \to \Ab$ be a homological functor to the category of commutative groups, and write $K_n(A) = K(\Sigma^{-n} A)$.
With notation as above, there is a convergent spectral sequence
\[
E^r_{p q} \Rightarrow K_{p+q}(P),
\]
with the $E^2$-sheet $E^2_{p q} = H_p(K_q(P_\bullet))$.
\end{theo}

The Baum--Connes conjecture for groupoids allows us to compare $P$ and $A$. We are going to use the following fundamental result proved by J.-L. Tu.

\begin{theo}[{\cite{tu:moy}}]\label{thm:tu}
Suppose that $G$ has the Haagerup property.
Then there exists a proper $G$-space $Z$ with an open surjective structure morphism $Z \to X$, and a $G \ltimes Z$-C$^*$-algebra $P$ which is a continuous field of nuclear C$^*$-algebras over $Z$, and such that $P\simeq C_0(X)$ in $\KKK^G$. 
\end{theo}

For the case of $\cT = \KKK^G$, $\cS = \KKK^X$, $E = \Ind^G_X$, and $F = \Res^G_X$, we have the following.

\begin{prop}[\cite{valmako:part1}]\label{prop:base-space-generation}
Let $G$ be an étale groupoid with torsion free stabilizers and satisfying the conclusions of Theorem \ref{thm:tu}.
Any separable $G$-C$^*$-algebra $A$ belongs to the localizing subcategory generated by the objects $\Ind^G_X B$ for $C_0(X)$-algebras $B$.
\end{prop}

\subsection{Smale spaces}

Next let us recall basic definitions on Smale spaces, mostly following \cite{put:HoSmale}.

\begin{defi}
A \emph{Smale space} $(X,\phi)$ is given by a compact metric space $(X, d)$ and a homeomorphism $\phi\colon X \to X$ such that:
\begin{itemize}
\item there exist constant $0 < \ep_{X}$ and a continuous map
\[
\{(x,y) \in X \times X \mid d(x,y) \leq \ep_{X}\} \to X, \qquad (x, y) \mapsto [x, y]
\]
satisfying the \emph{bracket axioms}:
\begin{align*}
[ x, x ] &= x,&
[ x, [ y, z ] ] &= [ x, z],\\
[ [ x, y ], z ] &= [ x,z ],&
\phi([x, y]) &= [ \phi(x), \phi(y)],
\end{align*}
for any $x, y, z$ in $X$ when both sides are defined.
\item there exists $0<\lambda < 1$ satisfying the \emph{contraction axioms}:
\begin{align*}
[x,y] &=y \Rightarrow d(\phi(x),\phi(y)) \leq \lambda d(x,y),\\
[x,y] &=x \Rightarrow d(\phi^{-1}(x),\phi^{-1}(y)) \leq \lambda d(x,y),
\end{align*}
whenever the brackets are defined.
\end{itemize}
\end{defi}

Suppose $x \in X$ and $ 0 < \ep \leq \ep_{X}$.
We define the \emph{local stable sets} and the \emph{local unstable sets} around $x$ as
\begin{align*}
 X^{s}(x, \ep) & = \{ y \in X \mid d(x,y) < \ep, [y,x] =x \}, \\ 
X^{u}(x, \ep) & = \{ y \in X \mid d(x,y) < \ep, [x,y] =x \}.
\end{align*}
The bracket $[x,y]$ can be characterized as the unique element of $X^s(x,\epsilon) \cap X^u(y, \epsilon)$ when $2 d(x, y) < \epsilon < \epsilon_X$.
This means that, locally, we can choose coordinates so that (see \cite{put:notes}*{Theorem 4.1.4})
\[
[\mhyph,\mhyph]\colon X^u(x,\epsilon)\times X^s(x,\epsilon)\to X
\]
is a homeomorphism onto an open neighborhood of $x\in X$ for all $\epsilon$ smaller than some $\epsilon^\prime_X\leq \epsilon_X/2$.

A point $x\in X$ is called \emph{non-wandering} if for all open sets $U\subseteq X$ containing $x$ there exists $N\in \N$ with $U\cap \phi^N(U)\neq \emptyset$.
Periodic points are dense among the non-wandering points \cite{put:notes}*{Theorem 4.4.1}.
We say that $X$ is non-wandering if any point of $X$ is non-wandering.
\emph{We will set a blanket assumption that Smale spaces are non-wandering}.
This holds in virtually all interesting examples. 

It can be shown that any non-wandering Smale space $(X,\phi)$ can be partitioned in a finite number of $\phi$-invariant clopen sets $X_1,\dots,X_n$, in a unique way, such that $(X_k,\phi|_{X_k})$ is \emph{irreducible} for $k=1,\dots,n$ \cite{put:funct}. Irreducibility means that for every (ordered) pair $U,V$ of nonempty open sets in $X$, there exists $N\in \N$ such that $U\cap \phi^N(V)\neq \emptyset$.

\begin{exem}
\label{ex:SFT}
A fundamental example of Smale space is given by \emph{shift of finite type} (or \emph{topological Markov shift}).
This can be modeled by finite directed graphs, as follows.
A \emph{directed graph} $\cG=(\cG^0,\cG^1,i,t)$ consists of finite sets $\cG^0$ and $\cG^1$, called vertices and edges, and maps $i, t\colon \cG^1 \to \cG^0$.
Thus, each edge $e \in \cG^1$ represents a directed arrow from $i(e) \in \cG^0$ to $t(e) \in \cG^0$.
Then a shift of finite type $(\Sigma_\cG,\sigma)$ is defined as the space of bi-infinite sequences of paths
\[
\Sigma_\cG=\{e=(e_k)_{k\in\mathbb{Z}}\in (\cG^1)^\Z \mid t(e_k)=i(e_{k+1})\},
\]
together with the left shift map $\sigma(e)_k=e_{k+1}$.
The metric is defined by $d(e,f) = 2^{-n-1}$ for $e \neq f$, where $n$ is the largest integer such that $e_k = f_k$ for $\absv{k} \le n$.
In particular, $d(e,f) \le 2^{-1}$ means that $e,f$ share the central edge, i.e., $e_0=f_0$. Then we can define
\[
[e,f]=(\dots,f_{-2},f_{-1},e_0,e_1,e_2,\dots).
\]
The pair $(\Sigma_\cG,\sigma)$ is a Smale space with constant $\epsilon=1/2$.
\end{exem}

We are particularly interested in groupoids encoding the \emph{unstable equivalence relation} of Smale spaces.
Given $x,y\in X$, we say they are 
\begin{itemize}
\item \emph{stably equivalent}, denoted by $x \sim_s y$, if
\[
\lim_{n\to \infty} d(\phi^{n}(x),\phi^{n}(y)) = 0;
\]
\item \emph{unstably equivalent}, $x \sim_u y$, if
\[
\lim_{n\to \infty} d(\phi^{-n}(x),\phi^{-n}(y)) = 0.
\]
\end{itemize}
We denote the graphs of these relations as
\begin{align}\label{eq:grpsmale}
R^s(X,\phi)&= \{(x,y) \in X\times X \mid x \sim_s y\},\\\notag
R^u(X,\phi)&= \{(x,y) \in X\times X \mid y \sim_u y\},
\end{align}
and treat them as groupoids, with source, range, and composition maps given by
\begin{align*}
s(x,y) &=y,&
r(x,y) &=x,&
(x,y) \circ (w,z) &= (x,z)\quad \text{if $y=w$.}
\end{align*}
The class of $x\in X$ under the stable (resp.~unstable) equivalence relation is called the \emph{global stable} (resp.~\emph{unstable}) \emph{set}, and is denoted by $X^s(x)$ (resp.~$X^u(x)$).
They satisfy the following identities:
\begin{align}\label{eq:incrun}
X^s(x)=&\bigcup_{n\geq 0} \phi^{-n}(X^s(\phi^n(x),\epsilon)),\\
\label{eq:incrun2}
X^u(x)=&\bigcup_{n\geq 0} \phi^{n}(X^s(\phi^{-n}(x),\epsilon)),
\end{align}
for any fixed $\epsilon < \epsilon_X$.

This leads to locally compact Hausdorff topologies on the above groupoids \cite{put:algSmale}: consider the induced topology on
\[
G_s^n = \{ (x, y) \mid y \in \phi^{-n}(X^s(\phi^n(x),\epsilon))\},\quad G_u^n = \{ (x, y) \mid y \in \phi^{n}(X^u(\phi^{-n}(x),\epsilon))\}
\]
as subsets of $X \times X$.
Then $R^u(X, \phi)$ is the union of the increasing sequence $(G_u^n)_n$, with \emph{open} inclusion maps $G_u^n \to G_u^{n+1}$.
This makes $G = R^u(X, \phi)$ a locally compact Hausdorff groupoid.
Moreover, the \emph{Bowen measure} defines a Haar system on $G$.
Of course, analogous considerations make $R^s(X, \phi)$ a locally compact Hausdorff groupoid with a Haar system.

To get an étale groupoid, we can take a transversal $T \subset X$ and restrict the base space to $T$, putting $G|_T = G^T_T$.
A convenient choice is to take $T = X^s(x)$ for some $x \in X$, with the inductive limit topology from \eqref{eq:incrun}, which is an example of generalized transversal. Slightly generalizing this, for a subset $P\subseteq X$, we write $X^s(P)$ meaning the union of all $X^s(x)$'s for $x\in P$, with the disjoint union topology. Analogously we define $X^u(P)=\bigcup_{x\in P}X^u(x)$.
Let us put
\begin{align}\label{eq:R-s-and-R-u-def}
R^s(X,P)&= R^s(X,\phi)|_{X^u(P)},&
R^u(X,P)&= R^u(X,\phi)|_{X^s(P)}.
\end{align}

\begin{theo}[\cite{put:spiel}*{Theorem 1.1}]
\label{rem:ame}
The groupoids in \eqref{eq:R-s-and-R-u-def} are amenable.
\end{theo}

\subsection{Maps of Smale spaces}\label{subsec:smalemaps}

Given two Smale spaces $(X,\phi)$ and $(Y,\psi)$, a continuous and surjective map $f\colon X \to Y$ is called a \emph{factor map} if it intertwines the respective self-maps, i.e.,
\begin{equation}\label{eq:intertw}
f\circ \phi=\psi\circ f.
\end{equation}

Equation \eqref{eq:intertw} is enough to guarantee that $f$ preserves the local product structure.
In particular, there is $\epsilon_f>0$ such that both $[x_1,x_2]$ and $[f(x_1),f(x_2)]$ are defined and $f([x_1,x_2])=[f(x_1),f(x_2)]$ for all $x_1,x_2$ with $d(x_1,x_2)< \epsilon_f$.

\begin{prop}[\cite{put:notes}*{Lemma 5.1.11}]\label{prop:finun}
Let $f \colon X \to Y$ be a factor map of Smale spaces, and $y_0\in Y$ be a point whose preimage is finite.
Then, given $\epsilon>0$, there exists $\delta = \delta(z, \epsilon) > 0$ such that
\[
f^{-1}(Y^u(y_0,\delta))\subseteq\bigcup_{i=1}^N X^u(x_i,\epsilon)
\]
where we write $f^{-1}(y_0) = \{x_1,\dots, x_N\}$.
\end{prop}

The assumption is automatically satisfied when $(Y, \psi)$ is non-wandering.

\begin{defi}
A factor map $f\colon (X,\phi)\to (Y,\psi)$ is called \emph{$s$-resolving} if it induces an injective map from $X^s(x)$ to $Y^s(f(x))$ for each $x \in X$.
It is called \emph{$s$-bijective}, if moreover these induced maps are bijective.
\end{defi}

\begin{theo}[\cite{put:lift}*{Corollary 3}]\label{thm:shiftfactor}
Let $(X,\phi)$ be an irreducible Smale space such that $X^s(x,\epsilon)$ is totally disconnected for every $x\in X$ and $0<\epsilon<\epsilon_X$. Then there is an irreducible shift of finite type $(\Sigma,\sigma)$ and an $s$-bijective factor map $f\colon (\Sigma,\sigma) \to (X,\phi)$.
\end{theo}

\begin{theo}[\cite{put:notes}*{Theorem 5.2.4}]\label{thm:oi}
Let $f \colon X \to Y$ is an $s$-resolving map between Smale spaces.
There is a constant $N\geq 1$ such that for any $y\in Y$ there exist $x_1,\dots, x_n$ in $X$, with $n\leq N$, satisfying
\[
f^{-1}(Y^u(y))=\bigcup_{k=1}^n X^u(x_k).
\]
For any $y\in Y$ the cardinality of the fiber $f^{-1}(y)$ is less than or equal to $N$.
\end{theo}

Let us list several additional facts about $s$-resolving maps, which can be found in \cite{put:notes}.
First, if each point in $Y$ is non-wandering, then $f$ is $s$-bijective.
Second, the induced maps $X^s(x) \to Y^s(f(x))$ and $X^u(x) \to Y^u(f(x))$ are both continuous and proper in the inductive limit topology of the presentation in \eqref{eq:incrun} and \eqref{eq:incrun2}.
If, moreover, $f$ is $s$-bijective, the map $X^s(x) \to Y^s(f(x))$ is a homeomorphism.
Assume that $X$ and $Y$ are irreducible, and $P$ is an at most countable subset of $X$ such that no two points of $P$ are stably equivalent after applying $f$.
Then
\[
f\times f \colon R^u(X,P) \to R^u(Y,f(P))
\]
is a homeomorphism onto an open subgroupoid of $R^u(Y,f(P))$. 

\subsection{Putnam's homology for Smale spaces}

For any shift of finite type $(\Sigma,\sigma)$, the $K_0$-group $K_0(C^*(R^u(\Sigma,\sigma)))$ can be described as Krieger's \emph{dimension group} $D^s(\Sigma, \sigma)$ \cite{krieger:inv}.
This group is generated by the elements $[E]$ for compact open sets $E$ in the stable orbits in $\Sigma$.
We can restrict to a collection of stable orbits which form a generalized transversal, and also assume that $E$ is contained in a \emph{local} stable orbit as well \cite{val:smaleb}*{Lemma 1.3}.

Let $(Y,\psi)$ be an irreducible Smale space with totally disconnected stable sets.
Then there is an irreducible shift of finite type $(\Sigma,\sigma)$ and an $s$-bijective factor map $f\colon (\Sigma,\sigma) \to (Y,\psi)$.
Then, we obtain a Smale space $\Sigma_n$ by taking the fiber product of $(n+1)$-copies of $\Sigma$ over $Y$ and the diagonal action $\sigma_n$ of $\sigma$ on $\Sigma_n$, so that we have $\Sigma_0=\Sigma$.

Then we get a simplicial structure on the groups $(D^s(\Sigma_n, \sigma_n))_{n=0}^\infty$, whose face maps are induced by the maps $\delta^s_k\colon\Sigma_{n}\to \Sigma_{n-1}$ which delete the $k$-th entry of a point in $\Sigma_{n}$.
This yields a well defined map between the corresponding dimension groups, via the assignment $[E]\mapsto [\delta^s_k(E)]$.
This way the groups $D^s(\Sigma_\bullet, \sigma_\bullet)$ form a simplicial object, and the associated homology $H^s_\bullet(Y,\psi)$, called \emph{stable homology} of $(Y, \psi)$, does not depend on the choice of $f$ \citelist{\cite{put:HoSmale}*{Section 5.5}\cite{val:smaleb}}.

\section{Fibered products of groupoids}
\label{sec:pullback-res-groupoids}

We start by defining an appropriate notion of fibered product between groupoids which will be used in the following proofs.
\begin{defi}
\label{def:mult-fib-prod}
Let $\alpha \colon H \to G$ be a homomorphism of groupoids, and $n \ge 2$.
We define the $n$-th fibered product of $H$ with respect to $\alpha$ as the groupoid $H^{\times_G n}$, as follows:
\begin{itemize}
\item the object space is the set
\[
(H^{\times_G n})^{(0)} = \{(y_1, g_1, y_2, \ldots, g_{n-1}, y_n) \mid y_k \in H^{(0)}, g_k \in G_{\alpha(y_{k+1})}^{\alpha(y_{k})} \}
\]
\item the arrows from $(y_1, g_1, y_2, \ldots, g_{n-1}, y_n)$ to $(y_1', g_1', y_2', \ldots, g_{n-1}', y_n')$ are given by the $n$-tuples $(h_1, \ldots, h_n) \in H_{y_1}^{y_1'} \times \cdots \times H_{y_n}^{y_n'}$ such that the squares in
\[
\begin{tikzcd}
 \alpha(y_1') &  \alpha(y_2') \arrow[l,"g_1'"'] & \cdots \arrow[l,"g_2'"'] & \alpha(y_n') \arrow[l,"g_{n-1}'"']\\
 \alpha(y_1) \arrow[u,"\alpha(h_1)"] &  \alpha(y_2) \arrow[l,"g_1"] \arrow[u,"\alpha(h_2)"] & \cdots \arrow[l,"g_2"] & \alpha(y_n) \arrow[l,"g_{n-1}"] \arrow[u,"\alpha(h_n)"'] 
\end{tikzcd}
\]
are all commutative.
\end{itemize}
\end{defi}

(Of course, we can put $H^{\times_G 1} = H$).
We say that an arrow in $H^{\times_G n}$ is represented by the tuple $(h_1, g_1', h_2, \ldots, g_{n-1}', h_n)$ in the above situation.
This way we can think of $H^{\times_G n}$ as a subset of $H \times G \times \cdots G \times H$, and in the setting of topological groupoids this gives a compatible topology on $H^{\times_G n}$ (for example, local compactness passes to $H^{\times_G n}$).

\begin{rema}
The above definition makes sense for $n$-tuples of different homomorphisms $\alpha_k\colon H_k \to G$, so that we can define $H_1 \times_G \cdots \times_G H_n$ as a groupoid.
The case of $n = 2$ appears in \cite{cramo:hom}.
\end{rema}

We will need a slight generalization of Definition \ref{def:mult-fib-prod} falling under this more general setting, where $H_j = H$ except for one value $j = k$, with $H_k = G$.

\begin{defi}
In the setting of Definition \ref{def:mult-fib-prod}, define a groupoid $G \times_G H^{\times_G n}$ as follows:
\begin{itemize}
\item the object space is the set
\begin{equation*}
(G \times_G H^{\times_G n})^{(0)} = \{(g_0, y_1, g_1, y_2, \ldots, g_{n-1}, y_n) \mid y_k \in H^{(0)},\\
g_0 \in G_{\alpha(y_1)}, g_k \in G_{\alpha(y_{k+1})}^{\alpha(y_{k})} (k \ge 1)\}
\end{equation*}
\item a morphisms from $(g_0, y_1, g_1, y_2, \ldots, g_{n-1}, y_n)$ to $(g_0', y_1', g_1', y_2', \ldots, g_{n-1}', y_n')$ is given by $k \in G_{r g_0}^{r g'_0}$ and an $n$-tuple $(h_1, \ldots, h_n) \in H_{y_1}^{y_1'} \times \cdots \times H_{y_n}^{y_n'}$ such that the squares in
\[
\begin{tikzcd}
r g'_0 & \alpha(y_1') \arrow[l,"g_0'"'] &  \alpha(y_2') \arrow[l,"g_1'"'] & \cdots \arrow[l,"g_2'"'] & \alpha(y_n') \arrow[l,"g_{n-1}'"']\\
r g_0 \arrow[u,"k"] & \alpha(y_1) \arrow[u,"\alpha(h_1)"] \arrow[l,"g_0"] &  \alpha(y_2) \arrow[l,"g_1"] \arrow[u,"\alpha(h_2)"] & \cdots \arrow[l,"g_2"] & \alpha(y_n) \arrow[l,"g_{n-1}"] \arrow[u,"\alpha(h_n)"'] 
\end{tikzcd}
\]
are all commutative.
\end{itemize}
\end{defi}

Again we say that an arrow of $G \times_G H^{\times_G n}$ is represented by $(k, g_0', h_1, \ldots, h_n)$ in the above situation.
As in the case of $H^{\times_G n}$, this induces a compatible topology in the setting of topological groupoids.

\begin{prop}
\label{prop:mor-equiv-for-fib-prod-with-G}
Let $\alpha \colon H \to G$ be a homomorphism of topological groupoids.
Then $H^{\times_G n}$ and $G \times_G H^{\times_G n}$ are Morita equivalent as topological groupoids.
\end{prop}

\begin{proof}
Consider the space
\[
Z = \{ (g_0, h_1, g_1, h_2, \ldots, g_{n-1}, h_n) \mid (g_0, \ldots, g_{n-1}) \in G^{(n)}, \alpha(r h_k) = s g_{k-1} \}.
\]
We define a left action of $G \times_G H^{\times_G n}$ as follows. The anchor map is
\[
Z \to (G \times_G H^{\times_G n})^{(0)}, (g_0, h_1, \ldots, h_n) \mapsto (g_0, r h_1, g_1, \ldots, r h_n),
\]
and an arrow of $G \times_G H^{\times_G n}$ with source $(g_0, r h_1, g_1, \ldots, r h_n)$ acts by
\[
(k, g_0', h_1', \ldots, h_n') . (g_0, h_1, \ldots, h_n) = (g_0', h_1' h_1, g_1', \ldots, h_n' h_n).
\]
On the other hand, there is a right action of $H^{\times_G n}$ defined as follows.
The anchor map is
\[
Z \to (H^{\times_G n})^{(0)}, \quad (g_0, h_1, \ldots, h_n) \mapsto (s h_1, g_1', \ldots, s h_n), \quad (g_k' = \alpha(h_k)^{-1} g_k \alpha(h_{k+1})).
\]
An arrow of $H^{\times_G n}$ with range $(s h_1, g_1', \ldots, s h_n)$ acts by
\[
(g_0, h_1, \ldots, h_n) . (h_1'', g_1, h_2'', \ldots, h_n'') = (g_0, h_1 h_1'', g_1, \ldots, h_n h_n'').
\]

We claim that $Z$ is a bibundle implementing the Morita equivalence (compatibility with topology will be obvious from the concrete ``coordinate transform'' formulas).

Comparing between $Z \times_{(G \times_G H^{\times_G n})^{(0)}} Z$ and $Z \times_{(H^{\times_G n})^{(0)}} H^{\times_G n}$ amounts to comparison of pairs $(h_k, h_k')$ with $r h_k = r h_k'$ on the one hand, and the composable pairs $(h_k, h_k'') \in H^{(2)}$ on the other.
There is a bijective correspondence between the two sides, given by the coordinate transform $h_k' = h_k h_k''$. Comparing $Z \times_{(H^{\times_G n})^{(0)}} Z$ with $G \times_G H^{\times_G n} \times_{(G \times_G H^{\times_G n})^{(0)}} Z$ amounts to comparing:
\begin{itemize}
\item on the side of $Z \times_{(H^{\times_G n})^{(0)}} Z$: $((g_0, h_1), (g_0', h_1'))$ with $(g_0, \alpha(h_1), (g_0', \alpha(h_1')) \in G^{(2)}$ and $s h_1 = s h_1'$, and $(h_k, h_k') \in H^{(2)}$ with $s h_k = s h_k'$ for $k \ge 2$;
\item on the side of $G \times_G H^{\times_G n} \times_{(G \times_G H^{\times_G n})^{(0)}} Z$: $(k, g_0'') \in G^{(2)}$, $(h_1, h_1'') \in H^{(2)}$ with $s h_1 = s g_0''$, and $(h_k, h_k'') \in H^{(2)}$ for $k \ge 2$.
\end{itemize}
Again we have a bijective correspondence by $h_k' = h_k''^{-1}$, $g_0 = g_0'' \alpha(h_1)^{-1}$, and $g_0' = k g_0'' \alpha(h_1)^{-1}$.
\end{proof}

A slight generalization is obtained by considering the groupoid $H^{\times_G a} \times_G G \times_G H^{\times_G b}$ for $a, b \ge 0$. This is defined as $H^{\times_G (a + b + 1)}$ in Definition \ref{def:mult-fib-prod}, with the difference that $h_{a+1}$ is not in $H^{y'_{a+1}}_{y_{a+1}}$, and instead in $G^{\alpha(y'_{a+1})}_{\alpha(y_{a+1})}$.

\begin{prop}
\label{prop:mor-equiv-for-fib-prod-with-G-in-the-middle}
The groupoid $H^{\times_G a} \times_G G \times_G H^{\times_G b}$ is Morita equivalent to $H^{\times_G (a + b)}$.
\end{prop}

\begin{proof}
Recall the construction in the proof of Proposition \ref{prop:mor-equiv-for-fib-prod-with-G} for the Morita equivalence between $G \times_G H^{\times_G b}$ and $H^{\times_G b}$: we have the space
\[
Z = \{ (g_0, h_1, g_1, \ldots, h_b) \mid (g_0, \ldots, g_{b-1}) \in G^{(b)}, \alpha(r h_k) = r g_k \},
\]
which is a bimodule between these groupoids.
Based on this, put
\begin{multline*}
\tilde Z = \{ (h_1, g_1, h_2, \ldots, g_a, g_{a+1}, h_{a+1}, g_{i+2}, \ldots, h_{a+b}) \mid (g_1, \ldots, g_{a+b}) \in G^{(a+b)},\\
\alpha(r h_k) = r g_k  \: (k \le a), \alpha(r h_k) = s g_k \: (k > a)\}.
\end{multline*}
This has obvious ``composition'' actions of $H^{\times_G a} \times_G G \times_G H^{\times_G b}$ from the left and $H^{\times_G (a+b)}$ from the right. By a similar argument as before, we can see that $\tilde Z$ implements a Morita equivalence.
\end{proof}

\medskip
Next let us show the compatibility of fiber products and generalized transversals.

\begin{prop}
\label{prop:gen-transv-for-fib-prod}
Let $\alpha\colon H \to G$ be a homomorphism of topological groupoids, and $f\colon T \to H^{(0)}$ be a generalized transversal.
Consider the space
\[
\tilde T = \{(t_1, g_1, t_2, \dots, t_n)\mid t_k \in T, g_k \in G^{f(t_k)}_{f(t_{k+1})} \}
\]
with the induced topology from the natural embedding into $T^n \times G^{n-1}$. The map
\[
\tilde f \colon \tilde T \to (H^{\times_G n})^{(0)},\quad (t_1, g_1, t_2, \dots, t_n) \mapsto (f(t_1), g_1, f(t_2), \dots, f(t_n))
\]
is a generalized transversal for $H^{\times_G n}$.
\end{prop}

\begin{proof}
Let us check the conditions in Definition \ref{def:gen-transv}.
First, $\tilde T$ meets all orbits of $H^{\times_G n}$.
Indeed, if we take a point $(y_1, g_1, y_2, \ldots, g_{n-1}, y_n) \in (H^{\times_G n})^{(0)}$, we can find $t_k \in T$ and $h_k \in H^{f(t_k)}_{y_k}$ for $k = 1, \ldots, n$.
Then there are unique $g_k'$ such that $(h_1, \ldots, h_n)$ represents an arrow from $(y_1, g_1, y_2, \ldots, g_{n-1}, y_n)$ to $(f(t_1), g_1', \ldots, f(t_n))$.

Next, let us check the condition (Ar).
Thus, take an arrow $x$ represented by $(h_1, g_1, h_2, \ldots, g_{n-1}, h_n)$ with range $r x = (f(t_1), g_1, f(t_2), \ldots, g_{n-1}, f(t_n))$, open neighborhood $U_0$ of $x$, and another $V_0$ of $r x$.
We may assume that these neighborhoods are of the form
\begin{align*}
U_0 &= (U_1' \times U_1'' \times U_2' \times \cdots \times U_n') \cap H^{\times_G n}, && (U_k' \in \cO(h_k), U_k'' \in \cO(g_k))\\
V_0 &= (V_1' \times V_1'' \times V_2' \times \cdots \times V_n') \cap \tilde T, && (V_k' \in \cO(t_k), V_k'' \in \cO(g_k).
\end{align*}

Then, for each $k$ we can find $\tilde U_k \in \cO(h_k)$ with $\tilde U_k \subset U_k'$, $\tilde V_k \in \cO(t_k)$ with $\tilde V_k \subset V_k'$ realizing the condition (Ar).
We claim that
\[
U = (\tilde U_1 \times U_1'' \times \cdots \times \tilde U_n') \cap H^{\times_G n},\quad
V = (\tilde V_1 \times V_1'' \times \cdots \times \tilde V_n) \cap \tilde T
\]
do the job.
Indeed, if $y = (\tilde h_1, \tilde g_1, \cdots, \tilde h_n) \in U$, another element $z = (\tilde h_1', \tilde g_1', \cdots, \tilde h_n')$ as the same source as $y$ if and only if $s \tilde h_k = s \tilde h_k'$ and $f(\tilde h_k)^{-1} \tilde g_k f(\tilde h_{k+1}) = f(\tilde h_k')^{-1} \tilde g_k' f(\tilde h_{k+1}')$ hold for all $k$.
Moreover, $r z \in \tilde T$ if and only if $r h_k' \in f(T)$ for all $k$.
The elements $\tilde g_k'$ are determined by the $\tilde h_k'$, and we can find such $\tilde h_k'$ uniquely by condition (Ar) for $U_k'$ and $V_k'$.
\end{proof}

Suppose $f \colon T \to H^{(0)}$ is a generalized transversal for $H$ such that $\alpha f \colon T \to G^{(0)}$ is also a transversal for $G$.
Then $\alpha$ induces a homomorphism of étale groupoids from $H' = H|_{f(T)}$ to $G' = G|_{\alpha f (T)}$.

\begin{coro}
\label{cor:compat-betw-fib-prod-restr-and-gen-transv}
In the setting above, $H^{\times_G n}$ is Morita equivalent to $H'^{\times_{G'} n}$.
\end{coro}

\begin{proof}
The construction of Proposition \ref{prop:gen-transv-for-fib-prod} gives a generalized transversal for $\tilde f \colon \tilde T \to (H^{\times_G n})^{(0)}$.
The étale groupoid obtained by this is isomorphic to $H'^{\times_{G'} n}$.
\end{proof}

\medskip
Now, let $(Y, \psi)$ be a non-wandering Smale space with totally disconnected unstable sets, and $f\colon (\Sigma, \sigma) \to (Y, \psi)$ be an $s$-resolving (hence $s$-bijective) factor map from a shift of finite type. 

Let $\Sigma_{n}$ denote the fibered product of $n + 1$ copies of $\Sigma$ with respect to $f$.
Then $\sigma_n = \sigma \times \cdots \times \sigma|_{\Sigma_{n}}$ defines a Smale space, which is again a shift of finite type.
If $a = (a^0, \ldots, a^n)$ and $b = (b^0, \ldots, b^n)$ are points of $\Sigma_n$, they are unstably (resp.~stably) equivalent if and only if $a^k$ is unstably (resp. stably) equivalent to $b^k$ for all $k$.

\begin{theo}
\label{thm:groupoid-fib-prod-compar-unstb-groupoid-Smale-sp-fib-prod}
In the setting above, set $G = R^u(Y, \psi)$, $H = R^u(\Sigma, \sigma)$, and $\alpha = f \times f\colon H \to G$ be the induced groupoid homomorphism.
Then $H^{\times_G n+1}$ is Morita equivalent to $R^u(\Sigma_{n}, \sigma_{n})$ as a locally compact groupoid.
\end{theo}

We will apply this to the $s$-bijective maps from Theorem \ref{thm:shiftfactor}.
A key step is the following proposition, which is our first technical result.

\begin{prop}\label{prop:transversality}
Let $f\colon (X,\phi)\to (Y, \psi)$ be an $s$-resolving map of Smale spaces.
Suppose $a_0,\ldots, a_n$ in $X$ are points such that $f(a_0)\sim_u f(a_k)$ for all $k$.
Then there are points $b_0, \ldots, b_n$ in $X$ satisfying
\begin{align*}
a_k &\sim_u b_k,&
f(b_0) &= f(b_k)
\end{align*}
for $k=0,\dots,n$.
\end{prop}

\begin{proof}
By compactness, there is a sequence of nonnegative integers $(j_n)_n$, $j_n \to \infty$, such that $\psi^{-j_n}(f(a_0))$ converges to a point $y\in Y$.
Since $f(a_k)\sim_u f(a_0)$, we must have the convergence $\psi^{-j_n}(f(a_k))\to y$ for $k= 1, \dots,n$ as well.

By Theorem \ref{thm:oi}, $y$ has a finite preimage in $X$.
Let $\epsilon=1/2\cdot \mathrm{max}\{\epsilon_X,\epsilon_f\}$, and choose $\delta=\delta(y,\epsilon)$ given by Proposition \ref{prop:finun}, so that we have
\[
f^{-1}(Y^u(y, \delta)) \subseteq \bigcup_{x \in f^{-1}(y)} X^u(x,\epsilon).
\]
Choose a big enough $N$ such that $\psi^{-j_N}(f(a_k)$ and $\psi^{-j_N}(f(a_0)$ belong to the same local unstable set, and $d(\psi^{-j_N}(f(a_k)), y) < \delta$ holds for $k = 0, \dots, n$.
Since $\psi^{-j_N}(f(a_k)) = f(\phi^{-j_N}(a_k))$, we have $\phi^{-j_N}(a_k)\in X(x_k,\epsilon)$ where $x_k\in X$ is a point satisfying $f(x_k)=y$.

Because of the choice of $\epsilon$, we can construct $b'_k = [x_k, \phi^{-j_N}(a_k)]$, and $f$ commutes with the bracket of these points.
Then $z = f(b'_k)=[y,\psi^{-j_N}(f(a_k))]$ is independent of $k$ by our choice of $j_N$.
We also have $b'_k \sim_u \phi^{-j_N}(a_k)$ by construction.
Then $b_k=\phi^{j_N}(b'_k)$ is unstably equivalent to $a_k$, and $f(b_k)=\psi^{j_N}(z)$ does not depend on $k$.
\end{proof}

Let us note in passing that instead of the $s$-resolving condition one could assume something else that still implies that $f$ is finite-to-one for the above proof to work.

\begin{figure}[ht]
\begin{center}
\begin{tikzpicture}[dot-node/.style={circle,fill,inner sep=2pt,minimum size=2pt}]
\tikzstyle{axes}=[]
\begin{scope}[scale=1]
\node[dot-node,label=235:{$y$}] (y) at (0,0) {};
\node[dot-node,label=45:{$z$}] (z) at (0,1) {};
\node[dot-node,label=below:{$\psi^{-j_N}(f(a_0))$}] (psi-jN-f-a0) at (-1.2,1) {};
\node[dot-node,label=below:{$\psi^{-j_N}(f(a_k))$}] (psi-jN-f-ak) at (1.2,1) {};
\draw[<->] (0,-.5) -- (0,1.5);
\draw[<->] (-1.7,1) -- (1.7,1);
\end{scope}
\begin{scope}[scale=0.8,shift={(-5,2)}]
\node[dot-node,label=235:{$x_0$}] (x0) at (0,0) {};
\node[dot-node,label=45:{$b'_0$}] (bp0) at (0,1) {};
\node[dot-node,label=below:{$\phi^{-j_N}(a_0)$}] (phi-jN-a0) at (-1,1) {};
\draw[<->] (0,-.5) -- (0,1.5);
\draw[<->] (-1.5,1) -- (0.5,1);
\end{scope}
\begin{scope}[scale=0.8,shift={(5,2)}]
\node[dot-node,label=315:{$x_k$}] (xk) at (0,0) {};
\node[dot-node,label=135:{$b'_k$}] (bpk) at (0,1) {};
\node[dot-node,label=below:{$\phi^{-j_N}(a_k)$}] (phi-jN-ak) at (1,1) {};
\draw[<->] (0,-.5) -- (0,1.5);
\draw[<->] (1.5,1) -- (-0.5,1);
\end{scope}
\draw[->,dashed] (x0) edge[] (y);
\draw[->,dashed] (bp0) edge[] (z);
\draw[->,dashed] (phi-jN-a0) edge[] (psi-jN-f-a0);
\draw[->,dashed] (xk) edge[] (y);
\draw[->,dashed] (bpk) edge[] (z);
\draw[->,dashed] (phi-jN-ak) edge[] (psi-jN-f-ak);
\end{tikzpicture}
\caption{The configuration of points in the proof of Proposition \ref{prop:transversality}. The vertical direction represents the stable direction, while the horizontal direction represents the unstable direction.}
\label{fig:keyproof}
\end{center}
\end{figure}
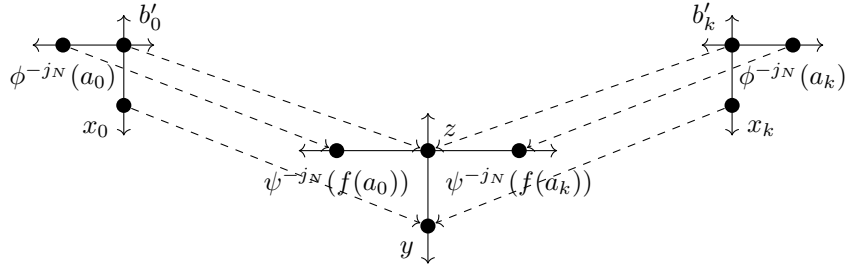

\begin{proof}[Proof of Theorem \ref{thm:groupoid-fib-prod-compar-unstb-groupoid-Smale-sp-fib-prod}]
We have an embedding of the groupoid $R^u(\Sigma_{n}, \sigma_{n})$ into $H^{\times_G n+1}$ by the correspondence
\[
(a^0, \ldots, a^n) \mapsto (a^0, \id_y, a^1, \ldots, \id_y, a^n) \quad (y = f(a^0) = \cdots = f(a^n))
\]
at the level of objects, and by
\[
((a^0, \ldots, a^n), (b^0, \ldots, b^n)) \mapsto ((a^0, b^0), \ldots, (a^n, b^n))
\]
at the level of arrows.
Proposition \ref{prop:transversality} implies that $\Sigma_{n} \subset (H^{\times_G n+1})^{(0)}$ meets all orbits of $H^{\times_G n+1}$.
Moreover, $a, b \in \Sigma_{n}$ are connected by an arrow in $(H^{\times_G n+1})^{(0)}$ if and only if they are connected in $R^u(\Sigma_{n}, \sigma_{n})$.
Thus, $R^u(\Sigma_{n}, \sigma_{n}) \curvearrowright (H^{\times_G n+1})^{\Sigma_{n}} \curvearrowleft H^{\times_G n+1}$ gives a Morita equivalence between the two groupoids.
It is a routine task to see that this is compatible with the topology on the two groupoids.
\end{proof}

\begin{rema}
Although we do not need it, we can replace $(\Sigma, \sigma)$ in Theorem \ref{thm:groupoid-fib-prod-compar-unstb-groupoid-Smale-sp-fib-prod} by another Smale space $(X, \phi)$ with totally disconnected stable sets.
The generalization of Proposition \ref{prop:transversality} with $X$ can be reduced to the above one as follows.
Fix a $s$-bijective factor map $f' \colon (\Sigma,\sigma) \to (X,\phi)$ from a shift of finite type.
Starting from $a^k \in X$ with unstably equivalent image in $Y$, taking inverse images $\bar a^k \in \Sigma$ of $a^k$, we can find points $\bar b^k \in \Sigma$ satisfying the assertion for the map $f \circ f'$.
Then the points $b^k = f'(\bar b^k)$ satisfy the assertion for $f$.
\end{rema}

Combining Proposition \ref{prop:mor-equiv-for-fib-prod-with-G}, Corollary \ref{cor:compat-betw-fib-prod-restr-and-gen-transv}, and Theorem \ref{thm:groupoid-fib-prod-compar-unstb-groupoid-Smale-sp-fib-prod}, we obtain the following.

\begin{theo}\label{thm:mormor}
In addition to $f\colon (\Sigma, \sigma) \to (Y, \psi)$ as above, let $f' \colon T \to \Sigma$ be a generalized transversal for the locally compact groupoid $R^u(\Sigma, \sigma)$ such that $f \circ f'\colon T \to Y$ defines a generalized transversal for $R^u(Y, \psi)$.
Denote the corresponding étale groupoids by 
\[
H = R^u(\Sigma, \sigma)|_{f'(T)},\qquad G = R^u(Y, \psi)|_{f \circ f'(T)}.
\]
The groupoid  $G \times_G H^{\times_G n+1}$ with respect to the natural inclusion $H \to G$ is Morita equivalent to $R^u(\Sigma_n, \sigma_n)$ as a topological groupoid.
\end{theo}

\subsection{Semidirect products and transversality}
\label{sec:semidir-prod-transv}

Let $\Gamma$ be a discrete group, and $G$ be a groupoid admitting a left action of $\Gamma$ by groupoid automorphisms $(\phi_\gamma)_{\gamma \in \Gamma}$.
We denote the corresponding semidirect product by $\Gamma \ltimes G$.
Thus, an arrow in $\Gamma \ltimes G$ is represented by pairs $(\gamma, g)$ for $\gamma \in \Gamma$ and $g \in G$, with range $\phi_\gamma(r(g))$ and source $s(g)$, and composition is given by
\[
(\gamma, g) (\gamma', g') = (\gamma \gamma', \phi_{\gamma'{-1}}(g) g') \quad (s(g) = \phi_{\gamma'}(r(g'))).
\]

Let $H$ be an another groupoid with an action of $\Gamma$ by automorphisms $(\psi_{\gamma})_{\gamma \in \Gamma}$, and let $\alpha\colon H \to G$ be a homomorphism commuting with $\phi$ and $\psi$.
Then we get a groupoid homomorphism $\Gamma \ltimes H \to \Gamma \ltimes G$.

\begin{prop}\label{prop:semidr-prod-transv-perm}
Suppose that the fiber product of base $H^{(0) \times_{G^{(0)}} n}$ meets all orbits in the groupoid pullback $H^{\times_G n}$.
Then $H^{(0) \times_{G^{(0)}} n}$ meets all orbits in $(\Gamma \ltimes H)^{\times_{\Gamma \ltimes G} n}$.
\end{prop}

\begin{proof}
Let $z = (y_1, (\gamma_1, g_1), y_2, \dots, (\gamma_{n-1}, g_{n-1}), y_n)$ be a point of $((\Gamma \ltimes H)^{\times_{\Gamma \ltimes G} n})^{(0)}$.
Note that we have
\[
\alpha(y_1) = \phi_{\gamma_1}(r(g_1)),\quad
\alpha(y_2) = s(g_1) = \phi_{\gamma_2}(r(g_2)), \quad \dots, \quad
\alpha(y_n) = s(g_{n-1}),
\]
thus the tuple
\[
z' = (\psi_{(\gamma_1 \dots \gamma_{n-1})^{-1}}(y_1), \phi_{(\gamma_2 \dots \gamma_{n-1})^{-1}}(g_1), \psi_{(\gamma_2 \dots \gamma_{n-1})^{-1}}(y_2), \dots, g_{n-1}, y_n)
\]
is a point in $(H^{\times_G n})^{(0)}$.
Moreover, the arrows
\[
(\gamma_1 \dots \gamma_{n-1}, \id_{\psi_{(\gamma_1 \dots \gamma_{n-1})^{-1}}(y_1)}) ,\quad \dots, \quad (\gamma_{n-1}, \id_{\psi_{\gamma_{n-1}^{-1}}(y_{n-1})}), \quad (e, \id_{y_n})
\]
in $\Gamma \ltimes H$ give an arrow from $z'$ to $z$ in $(\Gamma \ltimes H)^{\times_{\Gamma \ltimes G} n}$ up to the embedding $G \to \Gamma \ltimes G, g \mapsto (e, g)$, etc.
By assumption there is $z'' \in H^{(0) \times_{G^{(0)}} n}$ and an arrow from $z''$ to $z'$, hence $z$ is also in the orbit of $z''$.
\end{proof}

\section{Spectral sequence for Smale spaces}
\label{sec:approx-equivar-KK}

Let us start with the following generalization of Proposition \ref{prop:base-space-generation}.

\begin{theo}
\label{thm:KK-open-subgrpd-induction}
Suppose that $G$ is an étale groupoid with torsion-free stabilizers satisfying the conclusion of Theorem {\normalfont\ref{thm:tu}}, and that $H \subseteq G$ is an étale subgroupoid with the same base space $X$.
Any $G$-C$^*$-algebra in the category $\KKK^G$ belongs to the localizing subcategory generated by the image of $\Ind^G_H \colon \KKK^H \to \KKK^G$.
\end{theo}

Above, $H$ is an open subgroupoid of $G$ because $H^{(0)} = X$ and $H$ is étale.

\begin{proof}[Proof of Theorem {\normalfont \ref{thm:KK-open-subgrpd-induction}}]
Consider the functors
\begin{align*}
\Res^G_H\colon \KKK^G &\to \KKK^H,&
\Ind^G_H\colon \KKK^H &\to \KKK^G
\end{align*}
as in Section \ref{sec:ind-res-adj}.
By Proposition \ref{prop:simplicial-res-from-adj-fs}, we have a complementary pair $(\PI, \cN_\cI)$ for $\cI = \ker \Res^G_H$, with $\PI$ being generated by the image of $\Ind^G_H$ as a localizing subcategory.

Moreover, we have a natural isomorphism of functors $\Ind^G_X \simeq \Ind^G_H \Ind^H_X$.
Concretely, if $A$ is a $C_0(X)$-algebra, $\Ind^G_X A = C_0(G) \tens{s}{C_0(X)} A$ and $\Ind^G_H \Ind^H_X A= (C_0(G) \tensb{s}{C_0(X)}{r} C_0(H) \tens{s}{C_0(X)} A) \rtimes H$ are $G$-equivariantly strongly Morita equivalent via a Hilbert C$^*$-bimodule completion of $C_c(G) \tensb{s}{C_0(X)}{r} C_c(H) \tens{s}{C_0(X)} A$.
Combined with Proposition \ref{prop:base-space-generation}, we obtain that $A$ belongs to $\PI$.
\end{proof}

\begin{coro}\label{cor:obv}
Let $G$, $H$, and $A$ be as in Theorem {\normalfont\ref{thm:KK-open-subgrpd-induction}}. Let $P_H(A)\in \langle \Ind^G_H\KKK^H \rangle$ be the algebra appearing in the exact triangle 
\[
P_H(A) \to A \to N \to \Sigma P_H(A)
\]
that we get by applying Proposition {\normalfont\ref{prop:simplicial-res-from-adj-fs}}. Then we have $P_H(A)\simeq A$ in $\KKK^G$.
\end{coro}

\begin{coro}\label{cor:specseq}
Let $G$, $H$, and $A$ be as in Theorem {\normalfont\ref{thm:KK-open-subgrpd-induction}}.
Then we have a convergent spectral sequence
\begin{equation}\label{eq:spseqH}
E^2_{pq}=H_p(K_q(G \ltimes L^{\bullet+1} A)) \Rightarrow K_{p+q}(G \ltimes A),
\end{equation}
where $L^n A= (\Ind_H^G\Res^G_H)^n(A)$.
\end{coro}

\begin{proof}
The reduced crossed product functor
\[
\KKK^G \to \KKK, \quad A \mapsto G \ltimes A
\]
is exact and compatible with direct sums, while
\[
\KKK \to \Ab, \quad B \mapsto K_0(B)
\]
is a homological functor. Thus, their composition
\[
K_0(G \ltimes \mhyph) \colon \KKK^G \to \Ab
\]
is a homological functor, cf.~\cite{meyernest:tri}*{Examples 13 and 15}.
Now we can apply Theorem \ref{thm:spseqtri} to get a spectral sequence
\[
H_p(K_q(G \ltimes P_\bullet)) \Rightarrow K_{p+q}(G \ltimes P_H(A)),
\]
where $P_\bullet$ is a $(\ker \Res^G_H)$-projective resolution of $A$. The $(\ker \Res^G_H)$-projective resolution from Proposition \ref{prop:simplicial-res-from-adj-fs} gives the left hand side of \eqref{eq:spseqH}. Now the claim follows from Corollary \ref{cor:obv}.
\end{proof}

\begin{rema}
It would be an interesting question to cast the above constructions to groupoid equivariant $E$-theory~\cite{MR2044224}, since we mostly use formal properties of $\KKK^G$.
However, since some parts of our constructions involve reduced crossed products, there are some details to be checked.
(Note that $H$ need not be a proper subgroupoid.)
\end{rema}

\subsection{Projective resolution from subgroupoid induction}

Put $P_n = L^{n+1} A$ for the functor $L = \Ind^G_H \Res^G_H \colon \KKK^G \to \KKK^G$.
This is a simplicial object in $\KKK^G$, and the associated complex structure on $P_\bullet$ is given by the differential
\begin{equation}\label{eq:chcom}
\delta_n=\sum_{i=0}^n (-1)^i d^n_i \colon P_n \to P_{n-1},
\end{equation}
together with the augmentation morphism $\delta_0=\epsilon \colon P_0 = L A \to A$.
This makes $P_\bullet$ a $(\ker \Res^G_H)$-projective resolution of $A$ \cite{valmako:part1}*{Proposition 2.1}.

Suppose $G$ is a second countable locally compact Hausdorff étale groupoid, and $H$ is an open subgroupoid with the same base space.
Let us analyze the chain complex in \eqref{eq:chcom} more concretely.
Let $s_n \colon G^{(n)} \to X$ be the map $(g_1, \ldots, g_n) \mapsto s g_n$.

\begin{lemm}
Let $A$ be an $H$-C$^*$-algebra. The $C_0(G^{(n)})$-algebra $s_n^* A$ is endowed with a continuous action of the groupoid $G \times_G H^{\times_G n}$.
\end{lemm}

\begin{proof}
We use $(C_0(G) \otimes_{\min} A)_{\Delta(X)}$ as a model of $C_0(G) \otimes_{C_0(X)} A$, and analogous models for other relative C$^*$-algebra tensor products as well.
Recall that the arrow set of $G \times_G H^{\times_G n}$ can be identified with the set of tuples $(g, g_1, h_1,\dots,g_n,h_n)$ where $(g, g_1,\dots,g_n) \in G^{(2)}$, $h_i \in H$, and $s(g_i) = s(h_i)$.
Then
\[
C_0(G \times_G H^{\times_G n}) \tensor[^s]{\otimes}{_{C_0(G^{(n)})}} (C_0(G^{(n)}) \tens{s}{C_0(X)} A) \simeq (C_0(G \times_G H^{\times_G n} \times G^{(n)}) \otimes A)_Y,
\]
where $Y$ is the space of tuples $(g, g_1,h_1 \dots, g_n,h_n,g_1\dots,g_n,x )$ with $(g, g_1, h_1,\dots, g_n,h_n)$ as above and $x = s(g_n)$.
Let us set $K=G^{(n-1)}$ and think of $G^{(n)}$ as $K\tensor[_s]{\times}{_r} G$. We have
\[
C_0(G\times_G H^{\times_G (n-1)}\times_{K} G^{(n)} \tensb{s}{C_0(X)}{s} C_0(H) \tensor[^s]{\otimes}{_{C_0(G)}} A \simeq (C_0(G\times _G H^{\times_G (n-1)}\times_{K} G^{(n)} \times H) \otimes A)_Z,
\]
where $Z$ is the space of tuples $(g, g_1, h_1,\dots g_{n-1}, h_{n-1},g_n,h_n, x)$ where the components are related as above.
Via the obvious homeomorphism between $Y$ and $Z$, we have the identification of these algebras.
The structure map $\alpha\colon C_0(H) \tensor[^s]{\otimes}{_{C_0(G)}} A \to C_0(H) \tensor[^r]{\otimes}{_{C_0(G)}} A$ of the $H$-C$^*$-algebra induces an isomorphism onto
\[
C_0(G\times_G H^{\times_G (n-1)}\times_{K} G^{(n)} \tensb{s}{C_0(X)}{s} C_0(H) \tensor[^r]{\otimes}{_{C_0(G)}} A \simeq (C_0(G\times _G H^{\times_G (n-1)}\times_{K} G^{(n)} \times H) \otimes A)_{Z^\prime},
\]
where $Z'$ is the space of tuples $(g, g_1, h_1,\dots g_{n-1}, h_{n-1},g_n,h_n, y)$ as above and $y = r(h_n)$.
Finally, we have
\[
C_0(G\times_G H^{\times_G n}) \tensor[^r]{\otimes}{_{C_0(G^{(n)})}} (C_0(G^{(n)}) \otimes_{C_0(X)} A) = (C_0(G\times_G H^{\times_G n}\times G^{(n)} \otimes A)_{Y'},
\]
where $Y'$ is the space of tuples $(g, g_1, h_1, \dots g_n, h_n,g_1',\dots,g_n',y)$ as above, where $g_1' = g g_1 h_1^{-1}$, $g_i'=h_{i-1}g_ih_i^{-1}$ for $i>1$, and $y = s(g_n') = r(h_n)$.
Again the obvious bijection between $Y'$ and $Z'$ induces an isomorphism of the last two algebras, and combining everything we obtain an isomorphism
\[
C_0(G\times_G H^{\times_G n}) \tensor[^s]{\otimes}{_{C_0(G^{(n)})}} s_n^* A \to C_0(G\times_G H^{\times_G n}) \tensor[^r]{\otimes}{_{C_0(G^{(n)})}} s_n^* A
\]
which is the desired structure morphism of $G \times_G H$-C$^*$-algebra.
\end{proof}

\begin{prop}
\label{prop:compar-iter-ind-and-groupoid-fib-prod}
In the setting above, we have
\[
G \ltimes L^n A \simeq (G \times_G H^{\times_G n}) \ltimes s_n^* A.
\]
\end{prop}

\begin{proof}
 We have $L^n A = H^{\times_G n} \ltimes s_n^* A$ by expanding the definitions.
\end{proof}

Using the Morita equivalence between $G \times_G H^{\times_G n}$ and $H^{\times_G n}$, we can replace the formula above with $H^{\times_G n} \ltimes s_{n-1}^* A$.
This enables us to transport the simplicial structure on $(G \ltimes L^{n+1} A)_{n=0}^\infty$ to $(H^{\times_G (n+1)} \ltimes s_n^* A)_{n=0}^\infty$.
The proof is again straightforward from definitions.

\begin{prop}\label{prop:face}
The induced simplicial structure on $(H^{\times_G (n+1)} \ltimes s_n^* A)_{n=0}^\infty$ has face maps $d^n_i$ represented by the composition of $\KKK$-morphisms
\[
C^*_r(H^{\times_G (n+1)}, s_{n}^*A) \to C^*_r(H^{\times_G i} \times_G G \times_G H^{\times_G (n-i)}, s_{n}^*A) \to C^*_r(H^{\times_G n}, s_{n-1}^*A),
\]
where the first morphism is induced by the inclusion $H^{\times_G (n+1)} \to H^{\times_G i} \times_G G \times_G H^{\times_G (n-i)}$ as an open subgroupoid, and the second morphism is given by the Morita equivalence of Proposition {\normalfont\ref{prop:mor-equiv-for-fib-prod-with-G-in-the-middle}}.
\end{prop}

For a suitable choice of $P\subseteq \Sigma$ as in Subsection \ref{subsec:smalemaps}, we have an open inclusion of étale groupoid $f\times f\colon R^u(\Sigma,P)\rightarrow R^u(Y,f(P))$. We set $G= R^u(Y,f(P))$ and $H=(f\times f)(R^u(\Sigma,P))$. Notice that $G$ is an ample groupoid and $H$ is approximately finite-dimensional (AF) \citelist{\cite{put:notes}\cite{thomsen:smale}}.

\begin{prop}\label{prop:homsmale}
There is an isomorphism of chain complexes 
\begin{align*}
(K_0(G \ltimes L^{\bullet + 1} C_0(X)),\delta_\bullet) &\simeq (D^s(\Sigma_\bullet, \sigma_\bullet),d^s(f)_\bullet),&
(K_1(G \ltimes L^{\bullet + 1} C_0(X)),\delta_\bullet) &\simeq 0
\end{align*}
\end{prop}

Before going into the proof, let us recall the concept of \emph{correspondences} between groupoids.
In general, if $G$ and $H$ are topological groupoids, a correspondence from $G$ to $H$ is a topological space $Z$ together with commuting proper actions $G \curvearrowright Z \curvearrowleft H$, such that the anchor map $Z \to H^{(0)}$ is open (surjective) and induces a homeomorphism $G \backslash Z \simeq H^{(0)}$.
Of course, one source of such correspondence is Morita equivalence.
Another example is provided by continuous homomorphisms $f\colon G \to H$, where one puts $Z = \{[g, h] \mid f(s g) = r h \}$ with the relation $[g_1 g_2, h] = [g_1, f(g_2) h]$.

If $G$ and $H$ are (second countable) locally compact Hausdorff groupoids with Haar systems, a correspondence $Z$ induces a right Hilbert $C^*_r(H)$-module $C^*_r(Z)_{C^*_r(H)}$ with a left action of $C^*_r(G)$ \cite{MR1694789}.
If the action of $C^*_r(G)$ is in $\cK(C^*_r(Z)_{C^*_r(H)})$, we obtain a map $K_\bullet(C^*_r(G)) \to K_\bullet(C^*_r(H))$.
While finding a good characterization of this condition in terms of $Z$ seems to be somewhat tricky, in concrete examples as below it is not too difficult. On the other hand, composition of such Hilbert modules are easy to describe.
If $H'$ is another topological groupoid with Haar system, and $Z'$ is a correspondence from $H$ to $H'$, we have the identification
\[
C^*_r(Z)_{C^*_r(H)} \otimes_{C^*_r(H)} C^*_r(Z')_{C^*_r(H')} \simeq C^*_r(Z \times_H Z')_{C^*_r(H')}.
\]

\begin{proof}[Proof of Proposition \ref{prop:homsmale}]
As before, consider the functor $L = \Ind^G_H \Res^G_H \colon \KKK^G \to \KKK^G$.
Then we have
\begin{equation}\label{eq:crossed-prod-on-iter-comonad}
G \ltimes L^n A \simeq (G \times_G H^{\times_G n}) \ltimes s_n^* A,
\end{equation}
see \cite{valmako:part1}. By Theorem \ref{thm:mormor}, the C$^*$-algebra $G \ltimes L^{n+1} C_0(X)$ is strongly Morita equivalent to $C^*(R^u(\Sigma_n,\sigma_n))$. From this we have the identification of the underlying modules, and it remains to compare the corresponding simplicial structures. Let us give a concrete comparison of the maps $K_0(C^*_r(H^{\times_G n+1})) \to D^s(\Sigma_{n-1}, \sigma_{n-1})$ corresponding to the $0$-th face maps, as the general case is completely parallel.

Let us put $\tilde G = R^u(Y, \psi)$, $\tilde H = R^u(\Sigma, \sigma)$, and take a (generalized) transversal $T'$ for $R^u(\Sigma_n, \sigma_n)$, and put $K = R^u(\Sigma_n, \sigma_n)|_{T'}$, $K' = R^u(\Sigma_{n-1}, \sigma_{n-1})|_{\delta_0(T')}$ so that we have 
\begin{align*}
D^s(\Sigma_{n}, \sigma_{n}) &\simeq K_0(C^*_r(K)),&
D^s(\Sigma_{n-1}, \sigma_{n-1}) &\simeq K_0(C^*_r(K')).
\end{align*}
We denote the generalized transversal of $\tilde H^{\times_{\tilde G} n}$ induced by $P$, as in Proposition \ref{prop:gen-transv-for-fib-prod}, by $\tilde T_n$.

The map $\delta_0$ induces a groupoid homomorphism $K \to K'$, and hence a correspondence $Z_{\delta_0}$ from $K$ to $K'$. Composing this with the Morita equivalence bibundle ${}_{\tilde T_{n+1}} (\tilde H^{\times_{\tilde G} n+1})_{T'}$, we obtain a correspondence
\begin{equation}
\label{eq:corr-for-delta-0}
{}_{\tilde T_{n+1}} (\tilde H^{\times_{\tilde G} n+1})_{T'} \times_{K} Z_{\delta_0}
\end{equation}
from $H^{\times_G n+1}$ to $K'$ representing the effect of $\delta_0$ on the {$K$}-groups.

As for the $0$-th face map $d_0$ of $H^{\times_G \bullet+1}$, let $Z$ be the Morita equivalence bibundle between $G \times_G H^{\times_G n}$ and $H^{\times_G n}$ from Proposition \ref{prop:mor-equiv-for-fib-prod-with-G}.
Since $H^{\times_G n+1}$ is an open subgroupoid of $G \times_G H^{\times_G n}$, $Z$ becomes a correspondence from $H^{\times_G n+1}$ to $H^{\times_G n}$.
Composing this with the Morita equivalence ${}_{\tilde T_n} (\tilde H^{\times_{\tilde G} n})_{\delta_0(T')}$ between $H^{\times_G n}$ and $K'$, we obtain the correspondence
\begin{equation}
\label{eq:corr-for-d-0}
Z \times_{H^{\times_G n}} {}_{\tilde T_n} (\tilde H^{\times_{\tilde G} n})_{\delta_0(T')}
\end{equation}
from $H^{\times_G n+1}$ to $K'$ representing the effect of $d_0$.

It remains to check that the above correspondences are isomorphic, hence giving isomorphic Hilbert modules.
Expanding the ingredients of \eqref{eq:corr-for-d-0}, we obtain the space
\[
W = \{(g_0, h_1, g_1, h_2, \ldots, g_{n-1}, h_n) \mid (g_0, \ldots, g_{n-1}) \in G^{(n)}, h_k \in \tilde H^{s g_{k-1}}, (s h_1, \ldots, s h_n) \in \delta_0(T') \}.
\]
On the other hand, \eqref{eq:corr-for-delta-0} gives $W \times_K K'$ with
\begin{multline*}
W' = \{ (h_1, g_1, h_2, \ldots, g_n, h_{n+1}) \mid (g_1, \ldots, g_n) \in G^{(n)},\\
h_k \in \tilde H^{r g_k}, h_{n+1} \in \tilde H^{s g_n}, (s h_1, \ldots, s h_{n+1}) \in T' \}.
\end{multline*}
The operation $\mhyph \times_K K'$ ``kills'' the component $h_1$, and we obtain the identification with $W$.
\end{proof}

Thus, we obtain isomorphisms of homology groups
\begin{equation*}
H_p(K_q(G \ltimes L^{\bullet + 1} C_0(X)),\delta_\bullet)\simeq H^s_p(Y,\psi)\otimes K_q(\bC).
\end{equation*}
Combining this with Corollary \ref{cor:specseq} and Proposition~\ref{prop:homsmale}, we obtain the following main result of this paper.

\begin{theo}
\label{thm:putnam-homology-convergence}
Let $(Y,\psi)$ be an irreducible Smale space with totally disconnected stable sets. Then there is a convergent spectral sequence
\begin{equation}\label{eq:spec-seq-for-Putnam-homology}
E^r_{p q} \Rightarrow K_{p+q}(C^*(R^u(Y,\psi))),
\end{equation}
with $E^2_{pq} = E^3_{pq} = H_p^s(Y,\psi)\otimes K_q(\bC)$.
\end{theo}

\begin{coro}
The {$K$}-groups $K_i(C^*(R^u(Y,\psi)))$ have finite rank.
\end{coro}

\begin{proof}
By the above theorem, for $i = 0, 1$, the rank of $K_i(C^*(R^u(Y,\psi)))$ is bounded by that of $\bigoplus_k H^s_{i+2k}(Y,\psi)$.
The latter is of finite rank by \cite{put:HoSmale}*{Theorem 5.1.12}.
\end{proof}

\begin{rema}
By the Pimsner--Voiculescu exact sequence, the same can be said for the unstable Ruelle algebra $C^*(R^u(Y, \psi)) \ltimes_\psi \Z$.
If the stable relation $R^s(Y, \psi)$ also has finite rank $K$-groups, the Ruelle algebras will have finitely generated $K$-groups by \cite{MR3692021}.
\end{rema}

\section{Comparison of homologies}
\label{sec:compar-homology}

In fact, Putnam's homology is isomorphic to groupoid homology, as follows.

\begin{theo}\label{thm:isohom}
\label{thm:putnam-homology-is-groupoid-homology}
We have $H^s_\bullet(Y, \psi) \simeq H_\bullet(G,\Z)$.
\end{theo}

The homology groups on the right-hand side of Theorem \ref{thm:isohom} are a special case of groupoid homology with coefficients in equivariant sheaves \cite{cramo:hom}. 

When $G$ is ample (as is the case in the theorem), such $G$-sheaves are represented by \emph{unitary $C_c(G, \Z)$-modules} \cite{MR3270778}. Here, we consider the convolution product on $C_c(G, \Z)$, and a module $M$ over $C_c(G, \Z)$ is said to be unitary if it has the factorization property $C_c(G, \Z) M = M$. The correspondence is given by $\Gamma_c(U, F) = C_c(U, \Z) M$ for compact open sets $U \subset X$ if $F$ is the sheaf corresponding to such a module~$M$.

The \emph{homology of $G$ with coefficient in $F$}, denoted $H_\bullet(G, F)$, is the homology of the chain complex $(C_c(G^{(n)},\Z) \otimes_{C_c(X, \Z)} M)_{n=0}^\infty$ with differentials coming from the simplicial structure as above. Concretely, the differential is given by
\begin{align*}
\partial_n&\colon C_c(G^{(n)},\Z) \otimes_{C_c(X, \Z)} M \to C_c(G^{(n-1)},\Z) \otimes_{C_c(X, \Z)} M\\
\partial_n(f \otimes m) &= \sum_{i = 0}^{n-1} (-1)^i (d^n_i)_* f \otimes m + (-1)^n \alpha_n(f \otimes m),
\end{align*}
where $\alpha_n$ is the concatenation of the last leg of $C_c(G^{(n)},\Z)$ with $M$ induced by the module structure map $C_c(G, \Z) \otimes M \to M$.
This definition agrees with the one given in \cite{cramo:hom} as there is no need to take c-soft resolutions of equivariant sheaves (see \cite{valmako:part1}*{Proposition 1.8}).

More generally, if $F_\bullet$ is a chain complex of $G$-sheaves modeled by a chain complex of unitary $C_c(G,\Z)$-modules $M_\bullet$, we define $\bH_\bullet(G, F_\bullet)$, the \emph{hyperhomology} with coefficient $F_\bullet$, as the homology of the double complex $(C_c(G^{(p)},\Z) \otimes_{C_c(X, \Z)} M_q)_{p,q}$. As usual, a chain map of complexes of $G$-sheaves $f\colon F_\bullet \to F'_\bullet$ is a \emph{quasi-isomorphism} if it induces quasi-isomorphisms on the stalks.
When $F_\bullet$ and $F'_\bullet$ are bounded below, such maps induce an isomorphism of the hyperhomology \cite{cramo:hom}*{Lemma 3.2}.

\begin{proof}
Let us consider $G^{(n+1)}$ as an $H^{\times_G (n+1)}$-space by the anchor map
\[
(g_0, \ldots, g_n) \mapsto (g_1, \ldots, g_n) \in G^{(n)} = (H^{\times_G (n+1)})^{(0)}
\]
and the action map
\[
(h_1, g_1, h_2, \ldots, h_n) (g_0, \ldots, g_n) = (g_0 h_1^{-1}, g_1', \ldots, g_n')
\]
in the notation of Definition \ref{def:mult-fib-prod}.
Then $H_0(H^{\times_G (n+1)}, C_c(G^{(n+1)}, \Z))$ is a unitary $C_c(X, \Z)$-module by the action from the left, and the associated sheaf $F_n$ on $X$ is a $G$-sheaf by the left translation action of $G$.
At $x \in X$, the stalk can be presented as
\begin{equation}
\label{eq:stalk-Fn}
(F_n)_x = H_0(H^{\times_G (n+1)}, C_c((G^{(n+1)})^x, \Z)) = C_c((G^{(n+1)})^x, \Z)_{H^{\times_G (n+1)}}.
\end{equation}
Indeed, the sheaf corresponding to the $C_c(X, \Z)$-module $C_c(G^{(n+1)}, \Z)$ has the stalk $C^c((G^{(n+1)})^x, \Z)$ at $x$, and taking coinvariants by $H^{\times_G (n+1)}$ commutes with taking stalks.

We then have
\[
H_0(G, F_n) \simeq H_0(G \times_G H^{\times_G (n+1)},\Z) \simeq H_0(H^{\times_G (n+1)},\Z).
\]
The simplicial structure on $(G \times_G H^{\times_G (n+1)})_n$ leads to the complex of $G$-sheaves
\begin{equation}\label{eq:defputcp}
\dots \to F_2 \to F_1 \to F_0,
\end{equation}
and $H^s_\bullet(Y, \psi)$ is the homology of the complex obtained by applying the functor $H_0(G, \mhyph)$ to \eqref{eq:defputcp}.

We first claim that the augmented complex
\begin{equation}
\label{eq:sheaf-res-of-Z}
\dots \to F_2 \to F_1 \to F_0 \to \underline{\Z} \to 0
\end{equation}
is exact.
It is enough to check the exactness at the level of stalks.
In terms of the presentation \eqref{eq:stalk-Fn}, we have the chain complex
\[
\dots \to C_c((G^{(2)})^x, \Z)_{H^{\times_G 2}} \to C_c(G^x,\Z)_H \to \Z
\]
with differential
\begin{multline*}
d((g_1, \ldots, g_{n+1})) = (g_1 g_2, g_3, \ldots, g_{n+1}) - (g_1, g_2 g_3, \ldots, g_{n+1}) + \cdots \\
 + (-1)^{n-1} (g_1, \ldots, g_n g_{n+1}) + (g_1, \ldots, g_n),
\end{multline*}
where $(g_1, \ldots, g_n) \in (G^{(n)})^x$ represents the image of $\delta_{(g_1, \ldots, g_n)} \in C_c((G^{(n)})^x, \Z)$ in the coinvariant space, and the augmentation is given by $d(g) = 1$ at $n = 0$.
This has a contracting homotopy given by $\Z \to C_c(G^x,\Z)_H$, $a \to a (\id_x)$ and
\[
C_c((G^{(n)})^x, \Z)_{H^{\times_G n}} \to C_c((G^{(n+1)})^x, \Z)_{H^{\times_G n + 1}}, \quad
(g_1, \ldots, g_n) \to (\id_x, g_1, \ldots, g_n),
\]
hence \eqref{eq:sheaf-res-of-Z} is indeed exact.

We next claim that $H_k(G, F_n) = 0$ for $k > 0$.
Let $H_{n+1}$ be a subgroupoid of $G$ which is Morita equivalent to $H^{\times_G (n+1)}$ (this exists by choosing a transversal for $(\Sigma_{n}, \sigma_n)$ that maps bijectively to a transversal of $\Sigma$).
Then the module $H_0(H^{\times_G (n+1)}, C_c(G^{(n+1)}, \Z))$ representing $F_n$ is isomorphic to $H_0(H_{n+1}, C_c(G, \Z))$.
Thus, it is enough to check the claim when $n = 0$.

Let us write $M = H_0(H, C_c(G, \Z))$, and consider the double complex of modules $C_c(G^{(p + 1)} \times_X H^{(q)}, \Z)$ for $p, q \ge 0$, with differentials coming both from the simplicial structures on $(G^{(p)})_{p=0}^\infty$ and $(H^{(q)})_{q=0}^\infty$, cf. \cite{cramo:hom}*{Theorem 4.4}.
For fixed $p$, this is a resolution of $C_c(G^{(p)},\Z) \otimes_{C_c(X,\Z)} M$, hence the double complex computes $H_\bullet(G, F)$.
For fixed $q$, this is a resolution of $H_0(G, C_c(G \times_X H^{(q)}, \Z)) \simeq C_c(H^{(q)},\Z)$, and this double complex also computes $H_\bullet(H,\Z)$.
Since $H$ is Morita equivalent to an AF groupoid, $H_k(H, \Z) = 0$ by \cite{matui:hk}*{Theorem 4.11}.
We thus obtain $H_k(G, F_n) = 0$.

Finally, consider the hyperhomology $\bH_\bullet(G, F_\bullet)$.
On the one hand, by the above vanishing of $H_k(G, F_n)$, this is isomorphic to the homology of the complex $(H_0(G, F_n))_n$, i.e., $H^s_\bullet(Y, \psi)$.
On the other hand, since $F_\bullet$ is quasi-isomorphic to $\underline{\Z}$ concentrated in degree $0$, we also have $\bH_\bullet(G, F_\bullet) \simeq H_\bullet(G, \Z)$.
\end{proof}

We then have the following Künneth formula from the corresponding result for groupoid homology \cite{MR3552533}*{Theorem 2.4}.

\begin{coro}
Let $(Y_1, \psi_1)$ and $(Y_2, \psi_2)$ be Smale spaces with totally disconnected stable sets.
Then we have a split exact sequence
\[
0 \to \smashoperator[r]{\bigoplus_{a + b = k}} H^s_a(Y_1, \psi_1) \otimes H^s_b(Y_2, \psi_2) \to H^s_k(Y_1 \times Y_2, \psi_1 \times \psi_2) \to \smashoperator[r]{\bigoplus_{a + b = k-1}} \Tor(H^s_a(Y_1, \psi_1), H^s_b(Y_2, \psi_2)) \to 0.
\]
\end{coro}

\begin{rema}
As usual, the splitting is not canonical.
This generalizes \cite{MR3576278}*{Theorem 6.5}, in which one of the factors is assumed to be a shift of finite type.
Indeed, if $(Y_1, \psi_1)$ is a shift of finite type, the first direct sum reduces to $D^s(Y_1, \psi_1) \otimes H^s_k(Y_2, \psi_2)$, while the second direct sum of torsion groups vanishes as the dimension group $D^s(Y_1, \psi_1)$, being torsion-free, is flat.
\end{rema}

\begin{rema}
Theorem \ref{thm:putnam-homology-is-groupoid-homology} holds in general without the assumption of total disconnectedness on stable sets.
We plan to expand on this direction in a forthcoming work.
\end{rema}

\begin{rema}\label{rem:substitution-tiling}
Consider the Smale space $(\Omega, \omega)$ of a substitution tiling system.
As we observed in \cite{valmako:part1}*{Section 4}, if $G$ is an étale groupoid Morita equivalent to $R^u(\Omega, \omega)$, its groupoid homology $H_k(G, \Z)$ is isomorphic to the Čech cohomology $\check{H}^{d-k}(\Omega)$ for the constant sheaf $\underline{\Z}$ on $\Omega$.
Combined with Theorem \ref{thm:isohom}, we obtain
\[
H^s_k(\Omega,\omega)\simeq  \check{H}^{d-k}(\Omega),
\]
giving a positive answer to \cite{put:HoSmale}*{Question 8.3.2} in the case of tiling spaces.
\end{rema}

\section{K-theory of Ruelle algebra}
\label{sec:k-th-ruelle-alg}

\subsection{Semidirect product and homology}

Recall the following result for semidirect products.

\begin{prop}[\cite{valmako:part1}*{Proposition 3.8}]\label{prop:tfs}
Suppose that $\Gamma$ is torsion-free and satisfies the strong Baum--Connes conjecture, and that $G$ is an ample groupoid with torsion-free stabilizers satisfying the strong Baum--Connes conjecture.
Then any separable $(\Gamma \ltimes G)$-C$^*$-algebra $A$ belongs to the localizing subcategory generated by the image of $\Ind^{\Gamma \ltimes G}_X \colon \KKK^X \to \KKK^{\Gamma \ltimes G}$.
\end{prop}

Let $G$ and $\Gamma$ be as above, and let $A$ be a separable $(\Gamma \ltimes G)$-C$^*$-algebra.
We can choose an $(\ker \Res^G_X)$-projective resolution $P_\bullet$ of $A$ such that each $P_n$ is a $(\Gamma \ltimes G)$-C$^*$-algebra and the structure morphisms of $P_\bullet \to A$ are restriction of morphisms in $\KKK^{\Gamma \ltimes G}$.
For example, the standard choice
\[
P_n = (\Ind^G_X \Res^G_X)^{n+1} A = C_0(G^{(n+1)}) \otimes_{C_0(X)} A
\]
satisfies this assumption.

Then we have a complex of $\Gamma$-modules $K_q(G \ltimes P_\bullet)$.
We claim that the hyperhomology of $\Gamma$ with this coefficient is the $E^2$-sheet of our spectral sequence.

\begin{prop}
\label{prop:hyperhomology-spec-seq}
In the above setting, we have a spectral sequence
\[
E^r_{p q} \Rightarrow K_{p+q}(\Gamma \ltimes G \ltimes A)
\]
with $E^2_{p q} = \bH_p(\Gamma, K_q(G \ltimes P_\bullet))$.
\end{prop}

\begin{proof}
Put $L_G = \Ind^G_X \Res^G_X$ and $L_\Gamma B = C_0(\Gamma) \otimes B$.
Then $L_\Gamma$ is a model of the endofunctor $\Ind^{\Gamma \ltimes G}_G \Res^{\Gamma \ltimes G}_G$ on $\KKK^{\Gamma \ltimes G}$.
We then put $Q_{a,b} = L_\Gamma^{a+1} P_b$, and consider the natural bicomplex structure on $(Q_{a,b})_{a,b=0}^\infty$.

Let us show that the total complex $\Tot Q_{\bullet,\bullet}$ is a $(\ker \Res^{\Gamma \ltimes G}_X)$-resolution of $A$.
We have $Q_{a, b} \simeq \Ind^{\Gamma \ltimes G}_X \Res^{\Gamma \ltimes G}_X Q_{a-1,b-1}$, hence $Q_{a, b}$ is $(\ker \Res^{\Gamma \ltimes G}_X)$-projective.
Moreover, we claim that the augmented complex $\Tot Q_{\bullet,\bullet} \to A$ is $(\ker \Res^{\Gamma \ltimes G}_X)$-exact.
Indeed, after applying $\Res^{\Gamma \ltimes G}_G$ to the bicomplex $Q_{\bullet,\bullet}$, each row gives the complex $Q_{\bullet, b}$ consisting of the terms $C_0(\Gamma^{a+1}) \otimes P_b$ for $a \ge 0$.
Then the standard argument shows that $Q_{\bullet, b}$ is chain homotopic to $P_b$ concentrated at degree $0$, hence $\Res^{\Gamma \ltimes G}_G \Tot Q_{\bullet, \bullet}$ is chain homotopic to $P_\bullet$.

We now know that there is a spectral sequence converging to $K_{p+q}(\Gamma \ltimes G \ltimes A)$ with $E^2_{p q} = H_p(K_q(\Gamma \ltimes G \ltimes \Tot(Q_{\bullet,\bullet})))$.
Observe that $G \ltimes Q_{a,b} \simeq C_0(\Gamma^{a+1}) \otimes G \ltimes P_b$, where $\Gamma$ is acting by translation on the first coordinate of $\Gamma^{a+1}$.
Then we have
\[
K_q(\Gamma \ltimes G \ltimes Q_{a,b}) = \Z[\Gamma^a] \otimes K_q(G \ltimes P_b).
\]
Unpacking the differential, we see that $K_q(\Gamma \ltimes G \ltimes \Tot(Q_{\bullet,\bullet}))$ is the complex computing the hyperhomology as in the claim.
\end{proof}

\subsection{Application to Ruelle algebra}

Let $(Y, \psi)$ be a Smale space with totally disconnected stable sets.
Then the groupoid $\Z \ltimes_\psi R^u(Y, \psi)$ behind the \emph{unstable Ruelle algebra}
\[
\mathcal{R}_u(Y, \psi) = C^* (R^u(Y, \psi))\rtimes_\psi \Z
\]
fits into the setting we considered for semidirect products of étale groupoids (note the notational difference: $\mathcal{R}_u$ indicates the Ruelle algebra, while $R^u$ indicates the unstable equivalence relation). Indeed, as a generalized transversal of $R^u(Y, \psi)$ take $T = Y^s(P)$ for some set $P$ of periodic points of $\psi$, so $\psi$ induces an automorphism of the étale groupoid $G = R^u(Y, \psi)|_T$ as in \eqref{eq:R-s-and-R-u-def}. Then $\Z \ltimes_\psi G$ is Morita equivalent to $\Z \ltimes_\psi R^u(Y, \psi)$. 

Let $\pi\colon (\Sigma_\cG, \sigma) \to (Y, \psi)$ be a (regular) $s$-bijective map of Smale spaces.
Then we have the graph $\cG_N(\pi)$ modeling the $(N+1)$-fold fiber product of $\Sigma_\cG$ over $X$, so $\Sigma_{\cG_N(\pi)} = (\Sigma_\cG)_N$.
Combined with functoriality, we get simplicial groups
\[
(\BF(\gamma^s_{\cG_N(\pi)}))_{N=0}^\infty, \quad (\ker(1-\gamma^s_{\cG_N(\pi)}))_{N=0}^\infty.
\]

\begin{prop}
There is a spectral sequence $E^r_{pq} \Rightarrow K_{p+q}(\mathcal{R}_u(Y, \psi))$ with
\begin{align*}
E^1_{N, 2k} &= \BF(\gamma^s_{\cG_N(\pi)}),&
E^1_{N, 2k+1} &= \ker(1-\gamma^s_{\cG_N(\pi)}),
\end{align*}
with the $E^1$-differential $E^1_{p, q} \to E^1_{p-1,q}$ given by the simplicial structure.
\end{prop}

\begin{proof}
When $f\colon (\Sigma, \sigma) \to (Y, \psi)$ is an $s$-bijective map and $X' \subset \Sigma$ is a $\sigma$-invariant transversal such that $f(X') = X$, we get a subgroupoid $\Z \ltimes H \subseteq \Z \ltimes G$ from the reduction of $\Z \ltimes R^u(\Sigma, \sigma)$ to $X'$.
By Proposition \ref{prop:tfs}, $C_0(X)$ is in the triangulated subcategory generated by the image of $\Ind^{\Z \ltimes G}_{\Z \ltimes H}$.
We thus have a convergent spectral sequence
\[
E^1_{p q} = K_q((\Z \ltimes G) \ltimes (\Ind^{\Z \ltimes G}_{\Z \ltimes H} \Res^{\Z \ltimes G}_{\Z \ltimes H})^{p+1} C_0(X)) \Rightarrow K_{p+q}(C^*(\Z \ltimes G)).
\]
Moreover, by Proposition \ref{prop:semidr-prod-transv-perm}, we have
\[
E^1_{p q} \simeq K_q( C^* (R^u(\Sigma_p, \sigma_p))\rtimes \Z) \simeq K_q(\mathcal{R}_u(\Sigma_p, \sigma_p)).
\]
When $(\Sigma, \sigma)$ is presented by a graph $\cG$, these can be interpreted as $K_q(\mathcal{R}_u(\Sigma_{\cG_p(\pi)},\sigma))$, and the associated complex is the $E^1$-sheet of the above spectral sequence.

Finally, for any graph $\cG$ we have
\begin{align*}
K_0(\mathcal{R}_u(\Sigma, \sigma)) &\simeq \BF(\gamma^s_\cG),&
K_1(\mathcal{R}_u(\Sigma, \sigma)) &\simeq \ker(1-\gamma^s_\cG)
\end{align*}
for the connecting map $\gamma^s_\cG \colon \Z \cG^0 \to \Z \cG^0$, see Appendix \ref{sec:BF-grps}.
\end{proof}

There is another spectral sequence which plays nicer for low dimensional examples.
Consider the subsets $B_n \subset \cG_n(\pi)$ for normalization, which is a complete system of representatives of the free $S_{n+1}$-orbits.
We have an induced map $\gamma_{B_n}$ on $\Z B_n$, and the inductive limit groups $D_n = \varinjlim_{\gamma_{B_n}} \Z B_n$ form a chain complex that is quasi-isomorphic to $D^s(\cG_\bullet(\pi))$ \cite{put:HoSmale}*{Section 4.2}.
Then we have complexes
\begin{align*}
C_\bullet &= (\BF(\gamma_{B_n}))_{k=0}^\infty,&
C'_\bullet &= (\ker(1-\gamma_{B_n}))_{k=0}^\infty,
\end{align*}
endowed with induced differentials.

\begin{prop}
There is a spectral sequence $E^r_{p q} \Rightarrow K_{p+q}(\mathcal{R}_u(Y, \psi))$ such that $E^2_{p q} = 0$ for odd $q$, and the $E^2_{p q} = E^2_{p 0}$ for even $q$, sitting in a long exact sequence
\begin{equation}\label{eq:long-ex-seq-BF-homology}
\dots \to H_{p-1}(C'_\bullet) \to E^2_{p 0} \to H_p(C_\bullet) \to H_{p-2}(C'_\bullet) \to E^2_{p-1,0} \to \dots
\end{equation}
\end{prop}

\begin{proof}
We keep the notation $G$ and $X$ from previous proposition, but this time we use the spectral sequence given by Proposition \ref{prop:hyperhomology-spec-seq} for the resolution $P_n = C_0(G^{(n+1)})$ of $C_0(X)$.

As before, by transversality we have
\begin{align*}
K_0(G \ltimes P_n) &\simeq D^s(\cG_n(\pi)),&
K_1(G \ltimes P_n) &\simeq 0.
\end{align*}
As above, take the complex $D_\bullet$ with $D_n = \varinjlim_{\gamma_{B_n}} \Z B_n$.
The standard quasi-isomorphism from $D^s(\cG_\bullet(\pi))$ to $D_\bullet$ as given in \cite{put:HoSmale}*{Section 4.2} intertwines the actions of $\gamma_{\cG_n}$ and $\gamma_{B_n}$ by construction.
Generally, when $D'_\bullet \to D_\bullet$ is a quasi-isomorphism of complexes of $\Gamma$-modules, it induces an isomorphism of hyperhomology $\bH_p(\Gamma, D_\bullet) \simeq \bH_p(\Gamma, D'_\bullet)$.
We thus have
\[
E^2_{p q} = \bH_p(\Z, D^s(\cG_\bullet(\pi))) \simeq \bH_p(\Z, D_\bullet).
\]

It remains to put $\bH_p(\Z, D_\bullet)$ in a long exact sequence of the form \eqref{eq:long-ex-seq-BF-homology}.
Recall that the trivial $\Z[\Z]$-module $\Z$ has a free resolution of length $1$, given by
\[
\begin{tikzcd}
0 \ar[r] & \Z[\Z] \ar[r,"1-\tau_1"] & \Z[\Z] \ar[r] & \Z,
\end{tikzcd}
\]
where $\tau_1$ is the translation by $1$.
Thus, the hyperhomology $\bH_\bullet(\Z, D_\bullet)$ can be computed from the double complex with two rows of the form
\[
\begin{tikzcd}
 \dots \ar[r] & D_n \ar[d, "1-\gamma_{B_n}"] \ar[r] & D_{n-1} \ar[d, "1-\gamma_{B_{n-1}}"] \ar[r] & \dots \ar[r] & D_0 \ar[d, "1-\gamma_{B_0}"] \\
 \dots \ar[r] & D_n \ar[r] & D_{n-1} \ar[r] & \dots \ar[r] & D_0
\end{tikzcd}
\]
It contains a subcomplex given by
\[
\begin{tikzcd}
 \dots \ar[r] & \ker(1-\gamma_{B_n}) \ar[d] \ar[r] & \ker(1-\gamma_{B_{n-1}}) \ar[d] \ar[r] & \dots \ar[r] & \ker(1-\gamma_{B_0}) \ar[d] \\
 \dots \ar[r] & 0 \ar[r] & 0 \ar[r] & \dots \ar[r] & 0
\end{tikzcd},
\]
and the quotient is quasi-isomorphic to the complex of $(\BF(\gamma_{B_n}))_n$.
From this we indeed obtain a long exact sequence of the form \eqref{eq:long-ex-seq-BF-homology}.
\end{proof}

\begin{rema}
Recall that $Y$ is of the form $\varprojlim Y_0$ for a projective system of some compact metric space $Y_0$ and a suitable self-map $g\colon Y_0 \to Y_0$ as the connecting map at each step \cite{wie:inv}*{Theorem B}, analogous to the standard presentation of the $m^\infty$-solenoid in the next section.
Suppose further that $g$ is open and the groupoid C$^*$-algebra of the stable relation $R^s(Y, \psi)$ has finite rank $K$-groups.
Then, combining \cite{MR3692021}*{Theorem 1.1} and \cite{MR3868019}*{Corollary 5.5}, we see that $K_\bullet(\mathcal{R}_u(Y, \psi))$ fits in an exact sequence
\[
\begin{tikzcd}
K^{* + 1}(C(Y_0)) \arrow[r,"1-{[E_g]}"] & K^{* + 1}(C(Y_0)) \arrow[r] & K_\bullet(\mathcal{R}_u(Y, \psi)) \arrow[r] & K^\bullet(C(Y_0)) \arrow[r,"1-{[E_g]}"] & K^\bullet(C(Y_0)), 
\end{tikzcd}
\]
where $E_g$ is the $C(Y_0)$-bimodule associated with $g$.
It would be an interesting problem to compare the two ways to compute $K_\bullet(\mathcal{R}_u(Y, \psi))$.
\end{rema}

\section{Examples}
\label{sec:examples}

\subsection{Solenoid}
\label{sec:solonoid}

One class of examples is that of \emph{one-dimensional solenoids} \citelist{\cite{vd:sol}\cite{will:sol}}.
Let us first explain the easiest example, the $m^\infty$-solenoid.
Consider the space
\[
Y = \{ (z_0, z_1, \ldots) \mid z_k \in S^1, z_k = z_{k+1}^m \},
\]
which is the projective limit of
\begin{equation}\label{eq:dysolinv}
\begin{tikzcd}
S^1 & S^1 \arrow[l,"z^m \mapsfrom z"] & S^1 \arrow[l,"z^m \mapsfrom z"] & \,\cdots \arrow[l,"z^m \mapsfrom z"]
\end{tikzcd}.
\end{equation}
A compatible metric is given by
\[
d((z_k)_k, (z'_k)_k) = \sum_k m^{-k} d_0(z_k, z_k'),
\]
where $d_0$ is any metric on $S^1$ compatible with its topology; for example, one may take the arc-length metric $d_0(e^{i s}, e^{i t}) = \left|s-t\right|$ when $\left|s-t\right| \le \pi$.

There is a natural ``shift'' self-homeomorphism
\[
\phi \colon Y \to Y, \quad (z_0,z_1,\dots) \mapsto (z_0^m,z_1^m = z_0, z_2^m = z_1, \dots),
\]
with inverse given by $\phi^{-1}((z_0,z_1,\dots)) = (z_1, z_2, \ldots)$.
Then $(Y, \phi)$ is a Smale space \citelist{\cite{will:sol}\cite{put:HoSmale}}.

Denote by $\pi$ the canonical projection $Y \to S^1$ on the first factor.
As each step of \eqref{eq:dysolinv} is an $m$-to-1 map, $\pi^{-1}(z_0)$ can be identified with the Cantor set $\Sigma = \prod_{n= 1}^{\infty}\{0,1,\ldots,m-1\}$ for any $z_0 \in S^1$.
This allows us to write local stable and unstable sets around $z = (z_k)_k$, as
\begin{align}\label{eq:solrep}
Y^s(z,\epsilon)&= \pi^{-1}(z_0) \cong \Sigma,&
Y^u(z,\epsilon)&=\{ (e^{i t m^{-k}} z_k)_{k=0}^\infty \mid \left | t \right | < \delta_\epsilon\}
\end{align}
for small enough $\epsilon > 0$, with $\delta_\epsilon > 0$ depending on $\epsilon$.
Note that $\pi$ defines a fiber bundle with fiber $\Sigma$, and $Y^u(z,\epsilon) \times \Sigma \to Y$ corresponding to the bracket map gives local trivializations.

Now, the groupoid $R^u(Y, \phi)$ is the transformation groupoid $\R \ltimes_\alpha Y$ for the flow
\[
\alpha_t(z_0, z_1, \ldots) = (e^{i t} z_0, e^{i t m^{-1}} z_1, \ldots, e^{i t m^{-k}} z_k, \ldots) \quad (t \in \R).
\]
Restricted to the transversal $\pi^{-1}(1)$, we obtain the ``odometer'' transformation groupoid $\Z \ltimes_\beta \Sigma$, where $\Sigma$ is identified with $\varprojlim_k \Z_{m^k}$, and the generator $1 \in \Z$ acts by the $+1$ map on $\Z_{m^k}$.

There is a well-known factor map from the two-sided full shift on $m$ letters onto $(Y, \phi)$.
Namely, writing
\[
\Sigma' = \{0,1, \ldots, m-1\}^\Z = \{ (a_n)_{n=-\infty}^\infty \mid 0 \le a_n < m \},
\]
we have a continuous map $f \colon \Sigma' \to Y$ by
\[
f((a_n)_n) = (z_k)_{k=0}^\infty, \quad z_k = \exp\biggl( 2 \pi i \sum_{j=0}^\infty m^{-j-1} a_{j - k} \biggr).
\]
Then we have $f \sigma = \phi f$ for $\sigma \colon \Sigma' \to \Sigma'$ defined by $\sigma((a_n)_n) = (a_{n+1})_n$.

This allows us to compute all relevant invariants separately.
As for the $K$-groups, \cite{MR2050130} gives
\begin{align*}
K_0(C^* (R^u(Y, \phi))) &\simeq \Z\left[\frac1m\right],&
K_1(C^* (R^u(Y, \phi))) &\simeq \Z.
\end{align*}
(These can be also obtained through Connes's Thom isomorphism $K_\bullet(C^* (R^u(Y, \phi))) \simeq K^{\bullet + 1}(Y)$.)
As for groupoid homology, we have
\[
H_\bullet(\Z \ltimes_\beta \Sigma, \Z) \simeq H_\bullet(\Z, C(\Sigma, \Z))
\]
where right hand side is the groupoid homology of $\Z$ with coefficient $C(\Sigma, \Z)$ endowed with the $\Z$-module structure induced by $\beta$.
This leads to
\begin{align*}
H_0(\Z \ltimes_\beta \Sigma, \Z) &\simeq C(\Sigma, \Z)_{\beta} \simeq \Z\left[\frac1m\right],&
H_1(\Z \ltimes_\beta \Sigma, \Z) &\simeq C(\Sigma, \Z)^{\beta} \simeq \Z,
\end{align*}
with coinvariants and invariants of $\beta$, while $H_n(\Z \ltimes_\beta \Sigma, \Z) = 0$ for $n > 1$.
The computation for $H^s_\bullet(Y, \phi)$ will be more involved, but one finds \cite{put:HoSmale}*{Section 7.3} that
\begin{align*}
H^s_0(Y, \phi) &\simeq D^s(\Sigma', \sigma) \simeq \Z\left[\frac1m\right],&
H^s_1(Y, \phi) &\simeq \Z,
\end{align*}
and $H^s_n(Y, \phi) = 0$ for $n > 1$.
Thus the spectral sequence of Theorem \ref{thm:putnam-homology-convergence} collapse at the $E^2$-sheet, and there is no extension problem.

Now let us look at the unstable Ruelle algebra.
From the Pimsner--Voiculescu exact sequence we get
\begin{align*}
K_0(C^* (R^u(Y, \phi))\rtimes\Z) &\simeq \Z \oplus \Z/(1-m)\Z, &
K_1(C^* (R^u(Y, \phi))\rtimes\Z ) &\simeq \Z.
\end{align*}
Let us relate our spectral sequence to these groups.

As in \cite{put:HoSmale}*{Section 7.3}, there is an $s$-bijective map $\pi\colon \Sigma_\cG \to Y$ for the graph $\cG$ with one vertex and $m$ loops around it.
Since $\Sigma_{\cG^2} \to Y$ is regular, we use this to compute fiber products.
Now $\cG^2$ is the complete graph with $m$ vertices.
We thus have
\begin{align*}
\BF(\gamma_{B_0^2}) &\simeq \Z/(m-1)\Z,&
\ker(1 - \gamma_{B_0^2}) &\simeq 0,
\end{align*}

At the next step, $\cG^2_1(\pi)$ is a graph with $m+2$ vertices.
\begin{figure}[h]
\begin{tikzcd}
 \circ \arrow[loop left] & \circ \arrow[l] \arrow[d] \arrow[dl]\\
 \circ \arrow[u] \arrow[r] \arrow[ur] & \circ \arrow[loop right]
\end{tikzcd}
\caption{$\cG^2_1(\pi)$ for $m = 2$}
\label{fig:graph-G-2-1}
\end{figure}
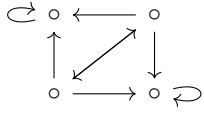
There is a single free $S_2$-orbit, as the nontrivial element $g \in S_2$ acts by flipping across the $45^\circ$ in Figure \ref{fig:graph-G-2-1}.
Thus the subset $B_1^2$ is a singleton (given by a choice of vertex with a loop), and the induced map $\gamma_{B_1^2}$ on $\Z B_1^2$ is the identity map.
We thus have
\begin{align*}
\BF(\gamma_{B_1^2}) &\simeq \Z,&
\ker(1 - \gamma_{B_1^2}) &\simeq \Z.
\end{align*}

For $n > 1$, we have $B^2_n = \empty$ thus $\BF(\gamma_{B_n^2}) = 0 = \ker(1 - \gamma_{B_n^2})$ for trivial reasons.
which are the torsion-free parts of $K_0(\Z \ltimes C^*(R^u(Y, \phi)))$.

Again from the computation \cite{put:HoSmale}*{Section 7.3} the boundary map $\BF(\gamma_{B_1^2}) \to \BF(\gamma_{B_0^2})$ is zero.
Then Proposition \ref{prop:hyperhomology-spec-seq} gives a spectral sequence with
\begin{align*}
E^2_{0 q} &\simeq \Z/(m-1)\Z,&
E^2_{1 q} &\simeq \Z,&
E^2_{2 q} &\simeq \Z,&
E^2_{p q} &\simeq 0 \quad (p > 2)
\end{align*}
for even $q$, and $E^2_{p q} = 0$ for odd $q$.
From this we get an exact sequence
\[
0 \to \Z/(m-1)\Z \to K_0(C^*(R^u(Y, \phi))) \to \Z \to 0
\]
and isomorphism $K_1(C^*(R^u(Y, \phi))) \simeq \Z$, as expected.

\subsection{Self-similar group action}
\label{sec:selfsim}

Let us explain an application to Smale spaces of self-similar group actions.
First let us recall the setting.
We refer to \citelist{\cite{MR2162164}\cite{nekra:crelle}} for unexplained terminology.
Let $X$ be a finite set of symbols, with $d = |X|$, and let $X^*$ denote the set of words with letters in $X$ with structure of a rooted tree.
A faithful action of a group $\Gamma$ on $X^*$ (by tree automorphisms) is \emph{self-similar} if it satisfies the condition
\[
\forall x \in X, \gamma \in \Gamma \quad \exists \gamma' \in \Gamma \ \forall w \in X^*\colon \gamma(x w) = \gamma(x) \gamma'(w).
\]
Writing $\gamma|_x = \gamma'$ in the above setting, we get a cocycle $(\gamma, x) \mapsto \gamma|_x$ for the action of $\Gamma$ on $X$.
This extends to a cocycle $\gamma|_u$ for $u \in X^*$.

Consider the full shift space $(\Sigma, \hat\sigma)$ where $\Sigma = X^\Z$ and $\hat\sigma(x)_k = x_{k-1}$ (note the opposite convention).
Thus, $x, y \in \Sigma$ are stably equivalent if and only if there is some $N$ such that $x_k = y_k$ holds for $k < N$, while they are unstably equivalent if and only if this holds for $k > N$.

Moreover, $x, y \in \Sigma$ are said to be \emph{asymptotically equivalent} if there is a finite set $F \subset \Gamma$ and a sequence $(\gamma_n)_{n \ge 0}$ in $F$, such that $\gamma_n(x_{-n} x_{-n+1} \dots) = y_{-n} y_{-n+1} \dots$ for $n \ge 0$.
The quotient space $S_{\Gamma, X}$ by this relation is called the \emph{limit solenoid}, and it is equipped with the homeomorphism induced by the shift map, denoted $\hat\sigma$. When $(\Gamma, X)$ is \emph{contracting}, \emph{recurrent} (or \emph{self-replicating}), and \emph{regular}, then $(S_{\Gamma, X},\hat\sigma)$ is a Smale space with totally disconnected stable sets \cite{nekra:crelle}*{Proposition 6.10} (alternatively, one can use \cite{nekra:crelle}*{Proposition 2.6} in conjunction with the main result of \cite{wie:inv}).

In the rest of the section we assume that $(\Gamma, X)$ satisfies the above assumptions.
Let us give a more detailed description of the ingredients of the spectral sequence.
Let us fix $x \in \Sigma$, and let $\Sigma^s(x)$ denote its stable equivalence class.
We denote by $T_0$ the set of points $y \in \Sigma^s(x)$ such that $y_n = x_n$ for $n < 0$.

\begin{prop}\label{prop:str-Sigma-s-x}
There is no point $y \neq x$ in $\Sigma^s(x)$ which is asymptotically equivalent to $x$.
\end{prop}

\begin{proof}
Suppose that $y \in \Sigma^s(x)$ is asymptotically equivalent to $y$, and take a finite subset $F \subset \Gamma$ satisfying the defining condition of asymptotic equivalence between $x$ and $y$ above.

By contractibility assumption there is a finite subset $\cN \subset \Gamma$ containing $e$, and $N_0 \in \N$ such that $\absv{u} \ge N_0$ implies $\gamma|_u \in \cN$ for $\gamma \in F$.
Moreover, by regularity and contractibility, there is $N_1 \in \N$ such that, for any $\gamma, \gamma' \in \cN$ and $u \in X^{N_1}$, one has either $\gamma(u) \neq \gamma'(u)$ or $\gamma|_u = \gamma'|_u$ \cite{nekra:crelle}*{Proposition 6.2}.
By assumption $x \sim_s y$, there is $N \in \N$ such that $x_k = y_k$ for $k \le -N$.
Set $M = N + N_0 + N_1$, and $u = x_{-M} x_{-M+1} \dots x_{-(N+N_1)} \in X^{N_0}$.
Then on one hand $\gamma' = \gamma_M |_u$ is in $\cN$, on the other $\gamma'$ should fix $v = x_{-(N+N_1)+1} \dots x_{-N} \in X^{N_1}$.
Thus we get $\gamma'|_v = e|_v$, hence $\gamma' = e$.
This implies that $x_k = y_k$ for $k > -N$, and consequently $x = y$.
\end{proof}

\begin{coro}
The projection map $(\Sigma, \hat\sigma) \to (S_{\Gamma, X}, \hat\sigma)$ is $s$-bijective.
\end{coro}

We can thus compute the homology groups $H^s_*(S_{\Gamma, X}, \hat\sigma)$ using the dimension groups of iterated fiber products of $\Sigma$ over $S_{\Gamma, X}$.

\begin{prop}\label{prop:hs0sgx-z1d}
We have $H^s_0(S_{\Gamma, X}, \hat\sigma) \simeq \Z[1/d]$.
\end{prop}

\begin{proof}
Since $D^s(\Sigma, \hat\sigma) \simeq \Z[1/d]$, it is enough to show that the boundary map $D^s(\Sigma_1, \hat\sigma_1) \to D^s(\Sigma, \hat\sigma)$ is trivial.
Note that $\Sigma_1$ is the graph of the asymptotic equivalence relation.

Recall that $\Sigma^s(x)$ admits a unique regular Borel measure $\mu$ which is invariant under the unstable equivalence relation and normalized as $\mu(T_0) = 1$.
Namely, with the identification $T_0 \simeq X^\N$, $\mu$ is the infinite product of uniform measure on $X$.
Then the isomorphism $D^s(\Sigma, \hat\sigma) \to \Z[1/d]$ is given by $[E] \mapsto \mu(E)$ for compact open sets $E \subset \Sigma^s(x)$.

Then it is enough to show that the asymptotic equivalence relation on $\Sigma^s(x)$ preserves $\mu$.
Let $y, z \in \Sigma^s(x)$ be two asymptotically equivalent points.
Thus, there is a sequence $(\gamma_n)_{n=0}^\infty$ in $\Gamma$ such that $\gamma_n(y_{-n} y_{-n+1} \dots) = z_{-n} z_{-n+1} \dots$ holds for all $n$.
For some fixed $N$, consider the  compact open neighborhood $E = \{y' \mid y'_k = y_k \: (k < N)\}$ of $y$ in $\Sigma^s(x)$, which has measure $\mu(E) = d^{-N}$.
Then
\[
\tilde{E} = \{(y', z') \mid y' \in E, \forall n \colon \gamma_n(y'_{-n} y'_{-n+1} \dots) = z'_{-n} z'_{-n+1} \dots \}
\]
is a compact open neighborhood of $(y, z)$ in $\Sigma_1^s((y, z))$, and the projection to first factor is a bijection $\tilde E \to E$.
Moreover, its projection to second factor, $E' = \{z' \mid \exists y' \colon (y', z') \in \tilde E \}$, can be written as $E' = \{z' \mid z'_k = z_k \: (k < N)\}$.
This has the same measure as $E$, hence the asymptotic equivalence relation indeed preserves $\mu$.
\end{proof}

\medskip
Let us next look at an étale groupoid model.
Let $T$ be the image of $T_0$ in $S_{\Gamma, X}$.
We denote the étale groupoid $R^u(S_{\Gamma, X}, \hat\sigma)|_T$ by $M_{\Gamma, X}$, so that $C^*(R^u(S_{\Gamma, X}, \hat\sigma))$ is strongly Morita equivalent to $C^*(M_{\Gamma, X})$.
As for the Ruelle algebra, the groupoid $D_{\Gamma, X} = \Z \ltimes M_{\Gamma, X}$ is the groupoid of germs of partially defined homeomorphisms on $X^\N$, of the form $x_1 \dots x_m z \mapsto x'_1 \dots x'_n \gamma z$ for $z \in X^\N$, with $x_i, x'_i \in X$ and $\gamma \in \Gamma$ (notice we range from $1$ to $m$ in the domain, and from $1$ to $n$ in the codomain).
The associated C$^*$-algebra $\cO_{\Gamma,X} = C^*(D_{\Gamma,X}) = \mathcal{R}_u(S_{\Gamma, X}, \hat \sigma)$ has a presentation similar to Cuntz--Pimsner algebras.
The structure of these algebras, including their $K$-groups and equilibrium states, have been previously studied in \cites{ExeHR:dilation,lrrw:eqstates}.

\begin{prop}[\cite{nekra:crelle}*{Proposition 6.8}]\label{prop:nekr-prop-6-8}
The image of $y \in T_0$ in $S_{\Gamma, X}$ is unstably equivalent to the image of $x$ if and only if $\gamma(x_k x_{k+1} \dots) = y_k y_{k+1} \dots$ for some $\gamma \in \Gamma$ and $k \ge 0$.
\end{prop}

This allows further recursive description of the groupoid homology, as follows.
Let $M_{\Gamma, X}^n$ the groupoid of germs of partially defined transformations on $X^\N$ of the form $x_1 \dots x_n z \mapsto x'_1 \dots x'_n \gamma z$ for $z \in X^\N$, with $x_i, x'_i \in X$ and $\gamma \in \Gamma$.
These form an increasing sequence of subgroupoids of $M_{\Gamma, X}$ such that $\bigcup_n M_{\Gamma, X}^n = M_{\Gamma, X}$.
By the self-similarity assumption, $M_{\Gamma, X}^n$ is isomorphic to the product groupoid $R_n \times M_{\Gamma, X}^0$, where $R_n$ is the complete equivalence relation on $X^n$.

\begin{prop}
The natural homomorphism from $H_0(M_{\Gamma, X}, \Z) \simeq \Z[1/d]$ to $K_0(C^*(M_{\Gamma, X}))$ is injective.
\end{prop}

\begin{proof}
First we have $H_0(M_{\Gamma, X}, \Z) \simeq \Z[1/d]$ by Proposition \ref{prop:hs0sgx-z1d} and Theorem \ref{thm:putnam-homology-is-groupoid-homology}.
By the above discussion, the image of this group in $K_0(C^*(M_{\Gamma, X}))$ is equal to the image of $K_0(M_{d^\infty}(\bC))$, where $M_{d^\infty}(\bC)$ is the UHF algebra that appears as the gauge invariant part of $\cO_d \subset \cO_{\Gamma, X}$.
The measure $\mu$ in the proof of Proposition \ref{prop:hs0sgx-z1d} defines a tracial state $\tau$ on $C^*(M_{\Gamma, X})$.
Thus, on the image of $H_0(M_{\Gamma, X}, \Z)$, $\tau$ induces isomorphism to $\Z[1/d]$.
\end{proof}

Turning to higher groupoid homology, we have
\[
H_k(M_{\Gamma, X}, \Z) \simeq \varinjlim_n H_k(M_{\Gamma, X}^n, \Z)
\]
from the compatibility of homology and increasing sequence of complexes.
Moreover, by the invariance of groupoid homology under Morita equivalence, we have
\[
H_k(M_{\Gamma, X}^n, \Z) \simeq H_k(M_{\Gamma, X}^0, \Z).
\]
Up to this isomorphism, the connecting maps $H_k(M_{\Gamma, X}^n, \Z) \to H_k(M_{\Gamma,X}^{n+1}, \Z)$ are all the same.

If moreover the action of $\Gamma$ on $X^\N$ is topologically free\footnote{This assumption is equivalent to injectivity of the associated virtual endomorphisms of $\Gamma$, and forces the group $\Gamma$ to be of polynomial growth \cite{MR2162164}*{Section 6.1}.}, then the groupoid $M_{\Gamma, X}^0$ (which is defined by the germs of the $G$-action on $X^\N$) can be identified with the transformation groupoid $\Gamma \ltimes X^\N$.
Note that $\Gamma \ltimes X^\N$ can be regarded as the projective limit of the groupoids $\Gamma \ltimes X^n$ for increasing $n$.
This again gives a presentation
\[
H_k(\Gamma \ltimes X^\N, \Z) \simeq \varinjlim_n H_k(\Gamma \ltimes X^n, \Z).
\]
By the recurrence assumption the groupoid $\Gamma \ltimes X^n$ is isomorphic to $R_n \times \Gamma$, and we have
\[
M_{\Gamma, X}^0 = M_{\Gamma, X}^k = M_{\Gamma, X}
\]
by \cite{nekra:crelle}*{Proposition 5.2}.
In particular, $H_k(M_{\Gamma, X}, \Z)$ becomes the inductive limit of the ordinary group homology $H_k(\Gamma, \Z)$, with respect to the map induced by the \emph{matrix recursion}
\[
\phi\colon \Z \Gamma \to \End_{\Z \Gamma}(\Z (X \cdot \Gamma)) \simeq \Z (R_1 \times \Gamma)
\]
for the \emph{permutation bimodule} $X \cdot \Gamma$, and the identification $H_k(\Gamma, \Z) \simeq H_k(R_n \times \Gamma, \Z)$.
Thus, under the above additional assumptions our spectral sequence gives a computation of the $K$-groups for the C$^*$-algebra $C^*(R^u(S_{\Gamma, X}, \hat\sigma))$ in terms of direct limits of sequences with terms $H_k(\Gamma, \Z)$.

Analogously, for $K$-theory we have
\begin{equation}\label{eq:self-similar-action-K-theory-ind-lim}
K_i(C^*(M_{\Gamma, X})) \simeq \varinjlim_n K_i(C^* (\Gamma \ltimes X^n)),
\end{equation}
cf.~\cite{nekra:crelle}*{Proposition 3.8}.

\begin{exem}
As a concrete example, let us look at the binary adding machine from \cite{MR2162164}*{Section 1.7.1}.
This one is given by $\Gamma = \Z$, $X = \{0, 1\}$, and the action of $1 \in \Z$ on $X^*$ is presented as
\begin{align*}
a(0 w) &= 1 w,&
a(1 w) &= 0 a(w)
\end{align*}
for $w \in X^*$.
To define the transversal we choose $x = (x_k)_k \in X^\Z$ with $x_k = 0$ for all $k$.
Then the induced equivalence relation on $T_0$ by Proposition \ref{prop:nekr-prop-6-8} agrees with the orbit equivalence relation of $\Z \ltimes \varprojlim \Z_{2^k}$ from the previous section (up to a change of convention of the shift map).
Turning to homology, on $H_0(\Z, \Z) \simeq \Z \simeq H_1(\Z, \Z)$, the matrix recursion induces multiplication by $2$, while it acts as the identity on $H_1(\Z, \Z)$.
This gives a direct computation of $H_0(M_{\Z, X}, \Z) \simeq \Z[1/2]$ and $H_1(M_{\Z, X}, \Z) \simeq \Z$, while $H_k(M_{\Z, X}, \Z) = 0$ for $k > 1$.
\end{exem}

Next let us look at higher dimensional odometer actions in more detail.
Let us take $B \in M_N(\Z) \cap \GL_N(\Q)$ with its eigenvalues all bigger than $1$.
Put $\Gamma = \Z^N$, and take a set $X$ of representatives of the $B \Gamma$-cosets in $\Gamma$.
Then we have a contracting, regular, and recurrent action of $\Gamma$ on $X^*$ such that the natural action of $\Gamma$ on $\varprojlim_n \Gamma / B^n \Gamma$ can be interpreted as the induced action on $X^\N$, see  \cite{MR2162164} for details. Thus, the groupoid homology $H_k(M_{\Gamma, X}, \Z)$ can be identified with
\[
H_k\Bigl(\Gamma, C\Bigl(\varprojlim_n \Gamma / B^n \Gamma, \Z\Bigr)\Bigr) \simeq \varinjlim_n H_k(\Gamma, C(\Gamma / B^n \Gamma, \Z)).
\]
The first step connecting map is $H_k(\Gamma, \Z) \to H_k(\Gamma, C(\Gamma / B \Gamma, \Z))$, induced by the $\Gamma$-equivariant embedding of $\Z$ into $C(\Gamma / B \Gamma, \Z)$ as constant functions.

Now, recall that $H_{N-k}(\Gamma, \Z)$ is isomorphic to $\medwedge^{k} \Gamma'$, where $\Gamma' = \Hom(\Gamma, \Z)$ is the group of homomorphisms from $\Gamma$ to $\Z$.
Let us describe an invariant formula for this.

Suppose that $(v_i)_{i=1}^N$ is a basis of $\Gamma$, and let $(v^i)_i$ be its dual basis in $\Gamma'$.
Then we get a graded commutative $\Z[\Gamma]$-algebra
\[
\Omega_{v_*}^\bullet = \Z[\Gamma] \otimes \medwedge^\bullet \{a(v^1), \dots, a(v^N)\},
\]
where $a(v^i)$ are formal symbols subject to the rule $a(v^i) \wedge a(v^j) = - a(v^j) \wedge a(v^i)$.
In particular, the degree $k$ part $\medwedge^k \{a(v^1), \dots, a(v^N)\}$ is a free commutative group with basis $a(v^{i_1}) \wedge \dots \wedge a(v^{i_k})$ for $1 \le i_1 < i_2 < \dots < i_k \le N$.
Then multiplication by the homogeneous element of degree $1$,
\[
D_{v_*} = \sum_{i=1}^N (1 - \lambda_{v_i}) \otimes a(v^i),
\]
defines a cochain complex structure on $\Omega_{v_*}^\bullet$.

The degree shifted chain complex $P_\bullet^{v_*}$ with $P_k^{v_*} = \Omega_{v_*}^{N-k}$ is a free resolution of $\Z$ as a $\Z[\Gamma]$-module.
Then $H_\bullet(\Gamma, \Z)$ can be computed as the homology of $\Z \otimes_\Gamma P_\bullet^{v_*}$, which is just the degree shift of $\medwedge^\bullet \{a(v^1), \dots, a(v^N)\}$ with trivial differential, hence is equal to $\medwedge^\bullet \{a(v^1), \dots, a(v^N)\}$ itself.

Let $f_{v_*} \colon \medwedge^\bullet \{a(v^1), \dots, a(v^N)\} \to \medwedge^\bullet \Gamma'$ the isomorphism of graded commutative rings  characterized by $a(v^i) \mapsto v^i$.
Then, if $(w_i)_i$ is another choice of basis in $\Gamma$, the map of resolutions $P_\bullet^{v_*} \to P_\bullet^{w_*}$ is compatible with $f_{v_*}$ and $f_{w_*}$ (up to chain-homotopy, this map is unique).

Now we are ready to identify the connecting map $\phi_k\colon H_k(\Gamma, \Z) \to H_k(\Gamma, C(\Gamma / B \Gamma, \Z)) \simeq H_k(\Gamma, \Z)$.

\begin{prop}\label{prop:higher-odom-groupoid-homolog}
With respect to the above identification $H_{N-k}(\Gamma, \Z) \simeq \medwedge^{k} \Gamma'$, the map $\phi_{N-k}$ is equal to $\medwedge^{k} B^t$, where $B^t\colon \Gamma' \to \Gamma'$ is the adjoint of $B$.
\end{prop}

\begin{proof}
Note that $C(\Gamma/ B \Gamma, \Z)$ can be identified with $\Ind^{\Gamma}_{B \Gamma} \Z$. Let us use the latter presentation to compute homology.
By the Smith normal form, there is a basis $(v_i)_i$ of $\Gamma$ and positive integers $(m_i)_i$ such that $(m_i v_i)_i$ is a basis of $B \Gamma$. Let us put $w_i = B^{-1} m_i v_i$, so that $(w_i)_i$ is another basis of $\Gamma$.
Then the cochain complex $\Ind^{\Gamma}_{B \Gamma} \Omega_{w_*}^\bullet$ is given by the same $\Z[\Gamma]$-module $\Omega_{w_*}^\bullet$ but the differential is the multiplication by
\[
D' = \sum_i (1 - \lambda_{m_i v_i}) \otimes a(w^i).
\]
Its degree shift $\Ind^{\Gamma}_{B \Gamma} P_\bullet^{w_*}$ gives a free resolution of $\Ind^{\Gamma}_{B \Gamma} \Z$.

Thus $H_k(\Gamma, \Ind^{\Gamma}_{B \Gamma} \Z)$ is computed from the complex $\Z \otimes_\Gamma \Ind^{\Gamma}_{B \Gamma} P_\bullet^{w_*}$, which is again the degree shift of $\medwedge^\bullet \{a(v^1), \dots, a(v^N)\}$ with trivial differential.
The identification
\[
H_k(\Gamma, \Ind^{\Gamma}_{B \Gamma} \Z) \simeq \medwedge^{N-k} \{a(v^1), \dots, a(v^N)\} \simeq H_k(\Gamma, \Z)
\]
is the canonical isomorphism $H_k(\Gamma, \Ind^{\Gamma}_{B \Gamma} \Z) \simeq H_k(\Gamma, \Z)$.
(One can see this by interpreting the $k$-th homology group $H_k(\Gamma, \Ind^{\Gamma}_{B \Gamma} \Z)$ as the $k$-th derived functor of $M \mapsto \Z \otimes_\Gamma \Ind^{\Gamma}_{B \Gamma} M$, and noting that $\Ind^{\Gamma}_{B \Gamma}$ is an exact functor sending projective modules to projective ones.)

We can concretely lift the embedding $\Z \to \Ind^{\Gamma}_{B \Gamma} \Z$ as a map of chain complexes $P^{v_*}_\bullet \to \Ind^{\Gamma}_{B \Gamma} P_\bullet^{w_*}$ as the $\Z[\Gamma]$-linear extension of
\[
a(v^{i_1}) \wedge \dots \wedge a(v^{i_k}) \mapsto \sum_{0 \le c_{i_j} < m_{i_j}} \lambda_{c_{i_1} v_{i_1} + \dots + c_{i_k} v_{i_k}} \otimes a(w^{i_1}) \wedge \dots \wedge a(w^{i_k}).
\]
Then the induced map $H_{N-k}(\Gamma, \Z) \to H_{N-k}(\Gamma, \Ind^\Gamma_{B \Gamma} \Z)$ becomes a transformation on $\medwedge^k \Gamma'$ characterized by
\[
v^{i_1} \wedge \dots \wedge v^{i_k} \mapsto m_{i_1} \dots m_{i_k} w^{i_1} \wedge \dots \wedge w^{i_k},
\]
which is exactly $\medwedge^k B^t$.
\end{proof}

\begin{coro}
In the above setting, we have
\[
H_{N-k}(M_{\Gamma, X}) \simeq \varinjlim_n \medwedge^k \Gamma'
\]
where the connecting map $\medwedge^k \Gamma' \to \medwedge^k \Gamma'$ at each step is given by $\medwedge^k B^t$.
\end{coro}

\begin{rema}
More generally, when $B_1, B_2, \dots$ are endomorphisms of $\Gamma$ of rank $N$, we get
\[
H_{N-k}\Bigl(\Gamma \ltimes \varprojlim_n \Gamma / B_1 \dots B_n \Gamma, \Z \Bigr) \simeq \varinjlim_n \medwedge^k \Gamma',
\]
where the right hand is with respect to the inductive system with connecting maps $\medwedge^k(B_1^t), \medwedge^k(B_2^t), \dots$ on $\medwedge^k \Gamma'$.
\end{rema}

Back to the higher dimensional odometer action of $\Gamma = \Z^N$, we have
\[
K_i(C^*( \Gamma \ltimes X^n)) \simeq \Z^{2^{N-1}} \simeq \bigoplus_k H_{i + 2k}(\Gamma, C(\Gamma / B^n \Gamma, \Z)),
\]
which implies that the spectral sequence for the $\Gamma$-action on the coefficient algebra $C(X^n)$ collapses at the $E^2$-sheet.
From Proposition \ref{prop:higher-odom-groupoid-homolog}, combined with the naturality of these spectral sequences, we obtain that the maps $K_i(C^*( \Gamma \ltimes X^n)) \to K_i(C^*( \Gamma \ltimes X^{n+1}))$ are injective.
Consequently, $K_i(C^*(M_{\Gamma, X}))$ has rank $2^{N-1}$ for $i = 0, 1$.
Since $\bigoplus_k H_{i + 2 k}(M_{\Gamma, X}, \Z)$ has the same rank and is torsion-free, we conclude that the spectral sequence for $M_{\Gamma, X}$ also collapses at the $E^2$-sheet.

\appendix
\section{Induction functor for subgroupoids}
\label{sec:app-ind-ftr}

Suppose that a groupoid $G$ acts freely and properly from the right on a second countable, locally compact, Hausdorff space $Y$.
Then the transformation groupoid $Y \rtimes G$ is Morita equivalent to the quotient space $Y / G$ as a groupoid, with $Y$ being the bibundle implementing the equivalence.
This induces the strong Morita equivalence between $G \ltimes C_0(Y) \simeq C^*(Y \rtimes G)$ and $C_0(Y / G)$.
In particular, for the case $Y = G$ and action given by right translation, we get an isomorphism between $G \ltimes C_0(G)$ and $\cK(L^2_r(G))$, where $L^2_r(G)$ is the right Hilbert $C_0(X)$-module completion of $C_c(G)$ with $C_0(X)$-module structure from $r^*\colon C_0(X) \to C_b(G)$ and inner product from the Haar system.

\begin{proof}[Proof of Proposition {\normalfont \ref{prop:Mor-equiv-A-IndGG-A}}]
As in the assertion, let $A$ be a $G$-C$^*$-algebra.
We have two actions of $G$: on the one hand, it acts on $s^* A$ by the combination of right translation on $G$ and the original action on $A$, while on the other hand it acts on $r^* A$ by the right translation on $G$ and trivially on $A$.
Then, the structure morphism $\alpha\colon s^*A \to r^*A$ of the action intertwines these two actions.
Morally $s^* A$ can be thought of as a space of sections $f(g) \in A_{s g}$ for $g \in G$, with the action of $G$ given by $f^{g'}(g) = g'^{-1} f(g g')$ for $(g, g') \in G^{(2)}$, while $r^* A$ as a space of sections $f(g) \in A_{r g}$ with $G$ acting by $f^{g'}(g) = f(g g')$ for $(g, g') \in G^{(2)}$.
We have $(\alpha f)(g) = g f(g)$ for the sections of the first kind, and these formulas give $(\alpha f^{g'})(g) = g f(g g') = (\alpha f)^{g'}(g)$.

Now, $\Ind_G^G\Res_G^G(A)$ is the crossed product of $s^* A$ by $G$, while $\cK(L^2_r(G)) \otimes_{C_0(X)} A$ is the crossed product of $r^* A$ by $G$.
Consequently we get an isomorphism between these algebras.
The extra action of $G$ on $\Ind_G^G\Res_G^G(A)$ comes from the action of $G$ on $s^* A$ given by the combination of the left translation on $G$ and the trivial action on $A$.
Under the above isomorphism, this corresponds to the action on $r^* A$ given by the combination of left translation on $G$ and the original action on $A$.
Thus, it corresponds to the diagonal action of $G$ on $\cK(L^2_r(G)) \otimes_{C_0(X)} A$.
\end{proof}

More generally, the same argument gives an isomorphism
\[
\phi \colon \Ind_H^G\Res^G_H A \simeq (C_0(G)\rtimes H) \tens{r}{C_0(X)} A,
\]
where $G$ acts diagonally on the algebra on the right. 

The functor $\Ind^G_H\colon \HCalg \to \GCalg$ preserves split extensions, respects equivariant Morita equivalence, and is compatible with homotopy.
Then, by the universal property of $\KKK^G$ [CITE] it extends to a functor $\Ind^G_H\colon \KKK^H \to \KKK^G$.
Let us give a more concrete description at the level of Kasparov cycles.

Consider an $H$-equivariant right Hilbert module $\cE$ over $B$.
By using an approximate unit in $B$, we can equip $\cE$ with a compatible $C_0(X)$-action. We can form the Hilbert module $C_0(G) \otimes \cE$ over $G_0(G) \otimes B$, and restrict on the diagonal to get $s^*\cE = (C_0(G) \otimes \cE)_{\Delta(X)}$ over $s^* B \simeq (C_0(G) \otimes B)_{\Delta(X)}$.
This still has an action of $H$, analogous to the right action of $H$ on $s^* B$.

Assume moreover $(\pi,\cE,T)$ is an equivariant Kasparov module between $H$-C$^*$-algebras. So $\cE$ is a graded right Hilbert module over $B$, $T$ is an odd adjointable (or self adjoint) endomorphism, and $\pi\colon C \to \cL(\cE)$ is a $*$-representation, with commutation relations as in \cite{gall:kk}.
Then $s^* \cE$ as a right Hilbert module over $s^* B$, with a left module structure over $s^* C$.
Moreover we can extend $T$ to $s^* T$ on $s^* \cE$ as the restriction of $1_{C_0(G)} \otimes T$, with the right commutation properties (they hold before restriction to $\Delta(X)$).
Finally, we take the crossed product by the right action of $H$,
\[
\Ind_H^G({\pi},\cE,T)= j_H(s^*\pi,s^*\cE,s^*T).
\]
This way, we obtain a map $\Ind_H^G\colon \KKK^H(C, B) \to \KKK^G(\Ind^G_H C, \Ind^G_H B)$, realizing the extension of $\Ind^G_H$ to $\KKK^H$.

\section{Bowen--Franks groups}
\label{sec:BF-grps}

Given $A \in \End(\Z^n)$, we call
\[
\BF(A) = \cok(1-A) = \Z^n/(1-A)\Z^n
\]
the \emph{Bowen--Franks group} of $A$.
For a square matrix $A \in M_n(\Z)$, we define $\BF(A)$ by interpreting $A$ as an element of $\End(\Z^n)$ by representing $v \in \Z^n$ by column vectors and computing $A v$ in the usual way.

Given $B \in \End(\Z^m)$ and $f \colon \Z^n \to \Z^m$ satisfying $B f = f A$, we have the induced maps $\BF(A) \to \BF(B)$, $\ker(1 - A) \to \ker(1 - B)$.
This correspondence is covariant.
Taking the transpose of $f$, we also have contravariant maps $\BF(B^t) \to \BF(A^t)$, etc.

Given a square matrix $A \in M_n(\{0, 1\})$, the associated shift of finite type $(\Sigma_A, \sigma)$ is given by
\begin{align*}
\Sigma_A &= \left\{ (x_k)_k \in \{1, \dots, n\}^\Z \mid A(x_k, x_{k+1}) = 1 \right\},&
\sigma(x)_k &= x_{k+1} \quad (x = (x_k)_k \in \Sigma_A),
\end{align*}
where we write $A_{i,j} = A(i,j)$.
Then we have
\begin{align*}
\Z \ltimes C^*(R^s(\Sigma_A, \sigma)) &\simeq \cO_A \otimes \cK,&
\Z \ltimes C^*(R^u(\Sigma_A, \sigma)) &\simeq \cO_{A^t} \otimes \cK
\end{align*}
for the Cuntz--Krieger algebra $\cO_A$, and we also have
\begin{align*}
K_0(\cO_{A^t}) &\simeq \BF(A) \simeq K^1(\cO_A),&
K_1(\cO_{A^t}) &\simeq \ker(I_n-A) \simeq K^0(\cO_A),
\end{align*}
see \citelist{\cite{MR561974}\cite{MR608527}\cite{MR1468312}}.

Given $A \in \End(\Z^n)$, let us look at the inductive system
\[
\begin{tikzcd}
\Z^n \arrow[r, "A"] & \Z^n \arrow[r, "A"] & \dots,
\end{tikzcd}
\]
and its inductive limit $\varinjlim_A \Z^n$.
Following \cite{put:HoSmale}, we represent the elements of $\varinjlim_A \Z^n$ by $[v, i]$ for $v \in \Z^n$ and $i \in \N$ (``$v$ at the $i$-th copy of $\Z^n$''), subject to the rules $[v, i] = [A v, i+1]$ and $[v, i] + [w, i] = [v + w, i]$.
We have $[v, i] = 0$ if and only if $A^j v = 0$ holds for some $j$.

We then have the automorphism $\alpha$ on $\varinjlim_A \Z^n$ defined by
\[
\alpha([v, i]) = [v, i+1].
\]

\begin{lemm}\label{lem:BF-from-alpha}
Under the above setting, we have
\begin{align*}
\cok(1 - \alpha) &\simeq \BF(A),&
\ker(1 - \alpha) &\simeq \ker(1 - A).
\end{align*}
\end{lemm}

\begin{proof}
The inductive limit functor is exact, hence we have
\begin{align*}
\ker(1 - \alpha) &\simeq \varinjlim_A \ker(1-A),&
\cok(1 - \alpha) &\simeq \varinjlim_A \BF(A).
\end{align*}
Thus it is enough to check that $A$ induces automorphisms on $\ker(1-A)$ and $\BF(A)$ (we will get the identity maps).

For $\ker(1-A)$, we have $v \in \ker(1-A)$ if and only if $v = A v$, hence $A\colon \Z^n \mapsto \Z^n$ restricts to bijective map (identity map) on $\ker(1-A)$.
For $\BF(A)$, again $[A v] = [v]$ holds in $\BF(A)$ hence $A$ induces the trivial isomorphism.
\end{proof}

Combined with the Pimsner--Voiculescu sequence, we get a direct computation of
\begin{align*}
K_0(\Z \ltimes C^*(R^s(\Sigma_A, \sigma))) &\simeq \BF(A),&
K_1(\Z \ltimes C^*(R^s(\Sigma_A, \sigma))) &\simeq \ker(1 - A).
\end{align*}

\medskip
Let $G$ be an oriented graph, again following the convention of \cite{put:HoSmale}.
Thus the associated shift of finite type is given by
\begin{align*}
\Sigma_G &= \{e = (e^k)_{k \in \Z} \mid e^k \in G^1, t(e^k) = i(e^{k+1}) \},&
\sigma(e)^k &= e^{k+1}.
\end{align*}
If we take the graph $G^2$, we have $\Sigma_{G^2} = \Sigma_{A}$ for the matrix $A \in M_{G^1}(\{0, 1\})$ defined by
\[
A_{e', e} = 
\begin{cases}
 1 & (t(e') = i(e) \Leftrightarrow (e', e) \in G^2)\\
 0 & \text{otherwise}.
\end{cases}
\]
Denoting the corresponding endomorphism on $\Z G^1$ by $\gamma^s_{G^2}$, we have
\[
D^s(G^2) = \varinjlim_{\gamma^s_{G^2}} \Z G^1.
\]
We have $D^s(G) \simeq D^s(G^2) \simeq D^s(G^3) \simeq \dots$ for the higher block presentations.

\begin{lemm}\label{lem:alpha-intertw}
Let $\alpha^s_{G^k}$ be the map $\alpha$ as above, induced by $\gamma^s_{G^k}$ on the inductive limit $D^s(G^k) = \varinjlim_{\gamma^s_{G^k}} \Z G^{k-1}$.
Then the isomorphism $D^s(G) \simeq D^s(G^k)$ intertwines $\alpha^s_G$ to $\alpha^s_{G^k}$.
\end{lemm}

\begin{proof}
This follows from the concrete form of isomorphism
\[
D^s(G) \simeq D^s(G^k), \quad [v, i] \mapsto [(t^k)^*(v), i],
\]
see \cite{put:HoSmale}*{Theorem 3.2.3}.
\end{proof}

Combining this with Lemma \ref{lem:BF-from-alpha}, we obtain $\BF(\gamma^s_G) \simeq \BF(\gamma^s_{G^k})$ and $\ker(1-\gamma^s_G) \simeq \ker(1 - \gamma^s_{G^k})$.

Finally, let $\pi\colon H \to G$ be a graph homomorphism that induces an $s$-bijective map $\Sigma_H \to \Sigma_G$.
Then there is an integer $K$ such that, for any $k$, the linear extension of
\[
\pi^{s,K}\colon H^{k-1} \to \Z G^{k+K-1}, \quad q \mapsto \sum_{q' \in (t^K)^{-1}(q)} \pi(q')
\]
implements the induced map $D^s(H) \to D^s(G)$ up to the identification $D^s(H) \simeq D^s(H^k)$ and $D^s(G) \simeq D^s(G^{k+K})$, see \cite{put:HoSmale}*{Section 3.4}.
Moreover, $\pi^{s,K}$ commutes with $\gamma^s_{H^k}$ and $\gamma^s_{G^{k+K}}$.
Combined with Lemmas \ref{lem:BF-from-alpha} and \ref{lem:alpha-intertw}, we see that $\pi^{s,K}$ also induces maps
\begin{align*}
\BF(\gamma^s_H) &\to \BF(\gamma^s_G),&
\ker(1-\gamma^s_H) &\to \ker(1-\gamma^s_G).
\end{align*}
These maps are again functorial for graph homomorphisms that induce $s$-bijective maps.

\raggedright

\end{document}